\documentclass[11pt,a4paper]{amsart}
\allowdisplaybreaks 

\usepackage{color}
\usepackage{amsmath}
\usepackage{amssymb}
\usepackage{amsthm}
\usepackage{amsfonts}
\usepackage{latexsym}
\usepackage{hyperref}
\usepackage{tikz}
\usetikzlibrary{matrix,arrows}

\newtheorem{theorem}{Theorem}[section]
\newtheorem{corollary}[theorem]{Corollary}
\newtheorem{lemma}[theorem]{Lemma}
\newtheorem{proposition}[theorem]{Proposition}
\newtheorem*{definition}{Definition}

\numberwithin{equation}{section}

\def \bC {\mathbb C}
\def \bD {\mathbb D}

\def \bR {\mathbb R}

\def \bR {\mathbb R}
\def \bR {\mathbb R}

\def \cD {\mathcal D}

\def \cF {\mathcal F}

\def \cH {\mathcal H}
\def \cI {\mathcal I}

\def \cL {\mathcal L}

\def \cP {\mathcal P}

\def \fg {\mathfrak g}
\def \fh {\mathfrak h}

\def \fk {\mathfrak k}

\def \fp {\mathfrak p}

\def \fS {\mathfrak S}

\def \fU {\mathfrak U}

\def \la {\lambda}
\def \ph {\varphi}

\def \lan {\langle}
\def \ran {\rangle}
\def \de {\partial}
\def \trans{\,{}^t\!}
\def \half{\frac12}
\def \inv{^{-1}}

\def \deg {\text{\rm deg\,}}

\title[Spherical analysis on homogeneous vector bundles]{Spherical analysis on homogeneous vector bundles}

\author[F. Ricci]{Fulvio Ricci}
\address{Scuola Normale Superiore, Piazza dei Cavalieri
7, 56126 Pisa, Italy } 
\email{{\tt fricci@sns.it}}

\author[A. Samanta]{Amit Samanta}
\address{Scuola Normale Superiore, Piazza dei Cavalieri
7, 56126 Pisa, Italy } 
\email{{\tt amit.gablu@gmail.com}}

\thanks{This research has been supported by the Italian MIUR PRIN Grant {\it Real and Complex Manifolds: Geometry, Topology and Harmonic Analysis}, 2010-2011.}



\begin{document}

\maketitle

\begin{abstract}
Given a Lie group $G$, a compact subgroup $K$ and a representation $\tau\in\widehat K$, we assume that the algebra of $\text{End}(V_\tau)$-valued,  bi-$\tau$-equivariant, integrable functions on $G$  is commutative. We 
present the basic facts of the related spherical analysis, putting particular emphasis on the r\^ole of the algebra of $G$-invariant differential operators on the homogeneous bundle $E_\tau$ over $G/K$. In particular, we observe that, under the above assumptions, $(G,K)$ is a Gelfand pair and
show that the Gelfand spectrum for the triple $(G,K,\tau)$ admits homeomorphic embeddings in $\bC^n$. 

In the second part, we develop in greater detail the spherical analysis for $G=K\ltimes H$ with $H$ nilpotent. In particular, for $H=\bR^n$ and $K\subset SO(n)$ and for the Heisenberg group $H_n$ and $K\subset U(n)$, we characterize the representations $\tau \in \widehat K$ giving a commutative algebra. \end{abstract}

\section*{Introduction}
\vskip.5cm

Let $(G,K)$ be a Gelfand pair with $G$ a Lie group and $K$ a compact subgroup of it. Recent work has put  attention on the fact that the spherical analysis on the bi-$K$-invariant algebra $L^1(K\backslash G/K)$ gains new interesting aspects from the fact that its Gelfand spectrum $\Sigma$, i.e., the space of bounded spherical functions with the compact-open topology, can be naturally embedded into some Euclidean space as a closed set \cite{F}. 

Such an embedding $\rho$ is defined by choosing a generating $k$-tuple  $(D_1,\dots,D_k)$ in the algebra of $G$-invariant differential operators on $G/K$ and assigning to the spherical function $\phi$ on $G/K$ the vector $\rho(\phi)=\big(\la_{D_1}(\phi),\dots,\la_{D_k}(\phi)\big)\in\bC^k$ whose entries $\la_{D_j}(\phi)$ are the eigenvalues of $\phi$ under the $D_j$'s.

This allows to introduce a notion of smoothness for functions defined on $\Sigma$, and to pose the problem, classical in Fourier analysis, of relating smoothness of the spherical transform of a given bi-$K$-invariant function on $G$ with properties of the function itself. This question has been investigated in detail for {\it nilpotent Gelfand pairs}, in which $G=K\ltimes H$ is a motion group on a nilpotent group $H$ \cite{FRY1, FRY2,FRY}.

In this paper we extend the basic framework for such analysis to the spherical transform of type $\tau$, where $\tau$ is an irreducible unitary representation of $K$ for which the appropriate commutativity assumptions are satisfied.

The notion of spherical transform of type $\tau$ goes back to \cite{G}. 
In most of the existing literature the accent is  on the case where $(G,K)$ is a symmetric pair.  For the general case we refer to \cite[Ch. 6]{War} and \cite{Ti, C}.

There are two equivalent ways to introduce the objects of our analysis on a triple $(G,K,\tau)$, where $G$ is a Lie group, $K$ a compact subgroup of it and $\tau\in\widehat K$ as above. 

In the first (or matrix-valued) picture one considers  the homogeneous bundle $E_\tau=G\times_\tau V_\tau$  with basis $G/K$ and linear operators on sections of $E_\tau$ commuting with the action of $G$. The Schwartz kernel theorem implies that, under mild continuity assumptions, any such operator can be represented by convolution with an $\textup{End}(V_\tau)$-valued\footnote{For $V,W$ finite dimensional complex vector spaces, $\textup{Hom}(V,W)$  denotes the space of linear operators from $V$ to $W$. If $V=W$, we write $\textup{End}(V)$ instead of $\textup{Hom}(V,V)$.} distributional kernel $F$ on $G$ satisfying the identity
\begin{equation}\label{F-invariance}
 F(k_1xk_2)=\tau(k_2^{-1})F(x)\tau(k_1^{-1}).
\end{equation}

The commutativity condition imposed on $\tau$ is that the algebra of $\textup{End}(V_\tau)$-valued integrable functions $F$ on $G$ satisfying the above identity is commutative with respect to the convolution defined in \eqref{def of conv} below. 

In the second  (scalar-valued) picture one considers the algebra of integrable scalar-valued functions $f$ on $G$ which are $K$-central and satisfy the identity
$$
f*\overline\chi_\tau=f,
$$
and the requirement on $\tau$ is that this algebra be commutative.

We say that $(G,K,\tau)$ is a {\it commutative triple} if either of these conditions is satisfied.

In the first part of the paper we recall definitions and basic facts about commutative triples and analyze in detail the relevant algebras of invariant differential operators in the two pictures, proving that they are finitely generated (Sections \ref{sect:diifop}).

It is  proved in \cite{D} that  Thomas's characterization \cite{T} of Gelfand pairs admits the appropriate extension to commutative triples  $(G,K,\tau)$ with $(G,K)$ a symmetric pair. This means that $(G,K,\tau)$  is commutative if and only if the algebra $\bD(E_\tau)$ of $G$-invariant differential operators on $E_\tau$ is commutative. In Theorem \ref{gelfand triple iff diff op commutative III} we provide a  proof of this property which applies to general triples, under the assumption that $G/K$ is connected.  Then the spherical functions of a commutative triple coincide with the joint eigenfunctions of these operators which satisfy \eqref{F-invariance} and take unit value at the identity of $G$.

From this equivalence we derive the following conclusion, which does not seem to appear in the literature: {\it if $(G,K,\tau)$ is a commutative triple, then $(G,K)$ is a Gelfand pair}. 
In particular, $G$ must be unimodular and the pair $(G,K)$ must fall in the classification of Lie Gelfand pairs \cite{V,Wo,Y}.

Since our proof is based on the analysis of $\bD(E_\tau)$, this result is limited to Lie groups $G,K$ with $G/K$ connected. It would be interesting to know if the same statement holds for commutative triples where $G$ is not a Lie group.

In Section \ref{sect:embeddings} we describe the embeddings of the Gelfand spectrum of the algebra of integrable functions satisfying \eqref{F-invariance} into Euclidean spaces, extending the result in \cite{F} to general $\tau$.

In Section \ref{sect:strong} we comment on the special case of a {\it strong Gelfand pair} $(G,K)$, defined by the condition that $(G,K,\tau)$ is commutative for {\it every} $\tau\in\widehat K$. Since $(G,K)$ is a strong Gelfand pair if and only if $(K\ltimes G,K)$ is a Gelfand pair
(with $K$ acting on $G$ by inner automorphisms), it is natural to compare the Gelfand spectrum of this pair with those of the individual triples $(G,K,\tau)$. We show that the former is the topological disjoint union of the latter ones.

In the second part of the paper we analyze the case where $G=K\ltimes H$ is a motion group on a Lie group $H$. The sections of $E_\tau$ are then identified with the $\textup{End}(V_\tau)$-valued functions on $H$ and the algebra of integrable functions satisfying \eqref{F-invariance} with the algebra of $\textup{End}(V_\tau)$-valued integrable functions on $H$ which transform according to $\tau$ under $K$:
$$
F(k\cdot x)=\tau(k)F(x)\tau(k\inv).
$$

In this context, the representation theoretical criterion for having a commutative triple can be reduced to conditions on the representations of $H$ rather than of $G$. At the same time, the spherical functions can be defined directly on $H$ and their properties described without lifting them to $G$, cf. Sections \ref{sect:bundles} and \ref{sect:diffonH}.
 
In Section \ref{sect:spectra} we show that the Gelfand spectrum $\Sigma_1$ of the pair $(G,K)$ is identified in a natural way with a quotient of the spectrum $\Sigma_\tau$ of the triple $(G,K,\tau)$, and that the quotient map becomes a canonical projection onto a coordinate subspace of $\bC^k$ when the two spectra are embedded in a compatible way in complex Euclidean spaces. 

From Section \ref{sect:positive} on we further restrict ourselves to the case where $H$ is nilpotent. We first extend to this kind of commutative triples the proof in \cite{B} that all bounded spherical functions are of positive type. 

In Section \ref{sect:H=Rn,Hn} we characterize the commutative triples with $H=\bR^n$ as those for which $\tau$ decomposes without multiplicities when restricted to the stabilizer of any point in $\bR^n$, and the commutative triples with $H$ equal to the Heisenberg group $H_n$ as those for which the tensor product of $\tau$ and the metaplectic representation restricted to $K$ decomposes without multiplicities. 

This allows to easily recognize the known fact  that 
$\big(SO(n)\ltimes\bR^n,SO(n)\big)$, $\big(U(n)\ltimes H_n,U(n)\big)$ are strong Gelfand pairs, and in addition  to classify the representations $\tau$ for which the triples $\big(SU(n)\ltimes\bC^n,SU(n),\tau\big)$,  $(SU(n)\ltimes H_n,SU(n),\tau)$ are commutative.

In Section \ref{sect:sphericalRnHn}  we give general formulas for the bounded spherical functions and finally, in Section \ref{sect:example}  , we explicitely compute them in the special  case $H=\bR^n$, $K=SO(n)$ and $\tau$ the natural representation on $\bC^n$.

\vskip.5cm

\section{Commutative triples and spherical functions of type $\tau$}\label{sect:triples}
\vskip.5cm

Let $G$ be a locally compact group, $K$ be a compact subgroup of $G$ and  $\tau$ a finite dimensional unitary  representation of $K$ on the  space $V_\tau$. 

By $C^\infty(G,V_\tau)$ we  denote the space of $V_\tau$-valued smooth functions on $G$ and  by $C^\infty_\tau(G,V_\tau)$ we denote  the subspace of functions $u$ satisfying the identity
\begin{equation}\label{section}
u(xk)=\tau(k^{-1})u(x),\quad\ \forall k\in K.
\end{equation}

Then $C^\infty_\tau(G,V_\tau)$ is naturally identified with the space of smooth sections of the homogeneous bundle $E_\tau=G\times_\tau V_\tau$. Similar notation will be used with $C^\infty$ replaced by other function (or distribution) spaces, like $\cD(=C^\infty_c),\cD',C^k,L^p$~etc.
 
It follows from the Schwartz kernel theorem that every linear operator, continuous with the respect to the standard topologies, mapping $\cD$-sections of $E_\tau$ into $\cD'$-sections of $E_\tau$ and commuting with the action of~$G$ on $E_\tau$, can be represented in a unique way as\footnote{The integral notation in \eqref{Schwartz-kernel-thm} and the pointwise identity \eqref{def of Hom(V)-valued algebra} must be appropriately interpreted when $F$ is not a function.}
\begin{eqnarray}\label{Schwartz-kernel-thm}
Tu(x)=\int_GF(y\inv x)u(y)dy,\quad u\in\cD_\tau(G,V_\tau),
\end{eqnarray}
with $F\in\cD'_{\tau,\tau}\big(G,\textup{End}(V_\tau)\big)$, i.e., 
\begin{eqnarray}\label{def of Hom(V)-valued algebra}
F(k_1xk_2)=\tau(k_2^{-1})F(x)\tau(k_1^{-1}).
\end{eqnarray} 

Conversely, any $F\in \cD'_{\tau,\tau}\big(G,\textup{End}(V_\tau)\big)$ defines a continuous $G$-invariant operator $T$ on $\cD$-sections of $E_\tau$ via formula \eqref{Schwartz-kernel-thm}.

In particular, operators $T=T_F$ as in \eqref{Schwartz-kernel-thm} with $F\in L^1_{\tau,\tau}\big(G,\textup{End}(V_\tau)\big)$ can be composed with each other (e.g., because they are bounded  on $L^1$-sections) and $T_{F_1}T_{F_2}=T_{F_1*F_2}$, where
\begin{eqnarray}\label{def of conv}
F_1*F_2(x)=\int_GF_2(y^{-1}x)F_1(y)dy. 
\end{eqnarray} 

\begin{definition}
$(G,K,\tau)$ is a {\it commutative triple} if the algebra $L_{\tau,\tau}^1(G,\textup{End}(V_\tau))$ is commutative. 
\end{definition}

The following well-known theorem (\cite{War}, Vol. II, Page-9, Prop. 6.1.1.6) gives a characterization of commutative triples in terms of representation of~$G$. By $\widehat{G}$ we denote the set of all equivalence classes of irreducible unitary representations of~$G$.

\begin{theorem}\label{representation theoretic criteria-general case}
$(G,K,\tau)$ is a commutative triple if and only if, for any $\pi\in\widehat{G}$, the multiplicity of $\tau$ in $\pi_{|_K}$ is at most~$1$.
\end{theorem}
\medskip

To an $\textup{End}(V_\tau)$-valued function $F$ we associate the scalar-valued function
\begin{equation}\label{S}
S_\tau F= d_\tau\textup{Tr}\,F.
\end{equation} 

The following statement is also well known \cite[vol. II]{War}.

\begin{proposition}
An $\textup{End}(V_\tau)$-valued function function $F$ satisfies \eqref{def of Hom(V)-valued algebra} if and only if $f=S_\tau F$ is $K$-central, i.e., 
\begin{eqnarray}\label{def of scalar valued algebra 1}
f(kxk^{-1})=f(x),\qquad\forall\,k\in K,
\end{eqnarray}
 and of type $\tau$, i.e.,
  \begin{eqnarray}\label{def of scalar valued algebra 2}
 f*(d_\tau\bar\chi_\tau m_K):= d_\tau\int_K f(xk)\chi_\tau(k)dk=f(x),  
 \end{eqnarray} 
 where $m_K$ denotes the normalized Haar measure on $K$. 
 
 Under these assumptions on $F$ and $f$, $S_\tau$ is bijective and the inverse map is given by
 \begin{equation}\label{Sinv}
S_\tau^{-1}f(x)=\int_K\tau(k)f(kx)dk. 
\end{equation}
\end{proposition}

By $L^p_\tau(G)^{\textup{int}K}$ we denote the space of $K$-central $L^p$ functions of type $\tau$. 
For $p=1$, $L^1_\tau(G)^{\textup{int}K}$ is closed under convolution. It is easy to verify that
$$
S_\tau\inv(f_1*f_2)=f_1*(S_\tau\inv f_2)=(S_\tau\inv f_1)*(S_\tau\inv f_2)\ ,
$$
which leads to the following conclusion.

\begin{proposition}
The map $S_\tau$ establishes an algebra isomorphism between $L^1_{\tau,\tau}\big(G,\textup{End}(V_\tau)\big)$ and $L^1_\tau(G)^{\textup{int}K}$.
\end{proposition}

\begin{definition}
Let $(G,K,\tau)$ be a commutative triple. A non-trivial function $\Phi\in L_{\tau,\tau}^\infty(G,\textup{End}(V_\tau))$ is said to be a $\tau$-\textit{spherical function}
if the map 
$$
F\rightarrow \widehat{F}(\Phi):=\frac{1}{d_\tau}\int_G\textup{Tr}\big[F(x)\Phi(x^{-1})\big]dx=\frac{1}{d_\tau}\textup{Tr}\big[F*\Phi(e)\big]
$$ 
is a homomorphism of $L_{\tau,\tau}^1(G,\textup{End}(V_\tau))$ into $\mathbb{C}$.
\end{definition} 

\begin{definition}
Let $(G,K,\tau)$ be a commutative triple. A non trivial function $\phi\in L^{\infty}_\tau(G)^{\textup{int}K}$ is said to be a \textit{trace} $\tau$-\textit{spherical function} if the map 
$$
f\rightarrow \widehat{f}(\phi):=\int_Gf(x)\phi(x^{-1})dx=f*\phi(e)
$$ is a homomorphism of $L^1_\tau(G)^{\textup{int}K}$ into $\mathbb{C}$.
\end{definition}

Observe that our definition of trace spherical function differs from that in \cite{War} by a factor of $d_\tau$.

The following theorem characterizes the $\tau$-spherical functions in terms of functional equations. We refer to \cite{GV} and \cite[Thm. 3.6]{C} for the proof.

\begin{theorem}\label{characterization of spherical function in terms of functional equation-1}
For $\Phi\in L_{\tau,\tau}^\infty(G,\textup{End}(V_\tau))$ and $\phi=\frac{1}{d_\tau^2}S_\tau(\Phi)$ the following are equivalent:
\begin{enumerate}
\item[\rm(i)] $\Phi$ is a $\tau$-spherical function.
\item[\rm(ii)]  $\phi$ is a trace $\tau$-spherical function.
\item[\rm(iii)]   $\Phi\in L_{\tau,\tau}^\infty(G,\textup{End}(V_\tau))$ is nontrivial and satisfies the functional equation
\begin{equation}\label{functional equation for Phi}
d_\tau\int_K\Phi(xky)\chi_\tau(k)dk=\Phi(y)\Phi(x). 
\end{equation}
\item[\rm(iv)]  $\phi\in L^\infty_\tau(G)^{\textup{int}K}$ is nontrivial and satisfies the functional equation 
\begin{eqnarray}\label{functional equation for phi}
\int_K\phi(xkyk^{-1})dk=\phi(x)\phi(y).
\end{eqnarray}
\end{enumerate}

A $\tau$-spherical function $\Phi$ satisfies $\Phi(e)=I$.
\end{theorem}

\vskip.5cm

\section{Differential operators on  homogeneous bundles}\label{sect:diifop}
\vskip.5cm

Let  $\mathbb{D}(G)$ be the algebra of  left-invariant differential operator on $G$. The action of $\mathbb{D}(G)$ on $C^\infty(G)$ induces an action of $\mathbb{D}(G)\otimes\textup{End}(V_\tau)$  on $C^\infty(G,V_\tau)= C^\infty(G)\otimes V_\tau$  given by
\begin{equation}\label{matrixD}
(D\otimes T)u:=D(Tu).
\end{equation}

With an abuse of notation, we will write $D$ for the ``scalar'' operator  $D\otimes I$.

The elements of $\mathbb{D}(G)\otimes\textup{End}(V_\tau)$ which preserve $C^\infty_\tau(G,V_\tau)$ are the ones which are invariant under all operators $\mu(k)$, $k\in K$ given by
\begin{equation}\label{D^K}
\mu(k)(D\otimes T)=D^{\textup{Ad}k}\otimes\tau(k^{-1})T\tau(k),
\end{equation} 
where, for an automorphism $\ph$ of $G$ and $D\in\bD(G)$,
$$
D^{\ph} f=\big(D(f\circ\ph)\big)\circ\ph\inv,
$$
cf. \cite[p. 120]{W}.

Denote by $\big(\mathbb{D}(G)\otimes \textup{End}(V_\tau)\big)^K$ the algebra of operators which are invariant under $\mu$ in \eqref{D^K}, and  by $\lambda:\fS(\fg)\longrightarrow \fU(\fg)\cong\bD(G)$  the symmetrization operator. The following statement is pretty obvious.

\begin{lemma}\label{Symm}
 The operator
$$
\Lambda:=\lambda\otimes I:\fS(\fg)\otimes\textup{End}(V_\tau) \longrightarrow \bD(G)\otimes\textup{End}(V_\tau)
$$
is a linear bijection from the space $\big(\fS(\fg)\otimes \textup{End}(V_\tau)\big)^K$ of $\textup{End}(V_\tau)$-valued polynomials $P$ on $\fg$ satisfying the identity
$$
P\circ\textup{Ad}(k)=\tau(k)\inv P\tau(k),
$$
onto
$\big(\mathbb{D}(G)\otimes \textup{End}(V_\tau)\big)^K$.
\end{lemma}
\medskip

The algebra $\bD(E_\tau)$ of $G$-invariant differential operators acting on smooth sections of the homogeneous bundle $E_\tau$ is then the quotient
$$
\bD(E_\tau)=\big(\mathbb{D}(G)\otimes \textup{End}(V_\tau)\big)^K\Big\slash\big(\mathbb{D}(G)\otimes \textup{End}(V_\tau)\big)_0^K\ ,
$$
 of $\big(\mathbb{D}(G)\otimes \textup{End}(V_\tau)\big)^K$ modulo the ideal of those operators which are trivial on $C^\infty_\tau(G,V_\tau)$.

In fact, we have the relations obtained by differentiating \eqref{section} with respect to $k$: for $X\in\fk$ and $u\in C^\infty_\tau(G,V_\tau)$,
\begin{equation}\label{X=-dtau}
Xu=-d\tau(X)u\ .
\end{equation}

The following statement, proved in \cite{M}, essentially says that the full ideal $\big(\mathbb{D}(G)\otimes \textup{End}(V_\tau)\big)_0^K$  is generated by the relations~\eqref{X=-dtau}.

\begin{theorem}\label{Lambda}
Let $\fp$ be an ${\rm Ad}(K)$-invariant complement of $\fk$ in $\fg$.
Then $\Lambda$ is a linear bijection from $\big(\fS(\frak{p})\otimes \textup{End}(V_\tau)\big)^K$ onto $\mathbb{D}(E_\tau)$.
\end{theorem}

Another consequence of \eqref{X=-dtau}, combined with the identity $\textup{End}(V_\tau)=d\tau\big(\fU(\fk)\big)$, is that any $\textbf{\textit {D}}\in\mathbb{D}(E_\tau)$  acts on $C^\infty_\tau(G,V_\tau)$ in the same way as a (scalar) element of 
$$
\mathbb{D}_K(G):=\{D\in\mathbb{D}(G):D^{\textup{Ad}(k)}=D\}=\{D\in\mathbb{D}(G):D^{R_k}=D\}.
$$

To be more precise, 
for a polynomial $P(x,y)\in \fS(\mathfrak{g})$ on $\fg$ ($x\in\mathfrak{k},y\in\mathfrak{p}$), define the modified symmetrization $\lambda^{\prime}(P)\in \mathbb{D}(G)$ as 
$$
\big(\lambda^{\prime}(P)f\big)(g)=P(\partial_x,\partial_y)_{|_{x=y=0}}f(g\exp y\exp x).
$$ 

It is quite obvious that $\lambda^\prime:\fS(\mathfrak{g})\rightarrow\mathbb{D}(G)$ is the linear bijection uniquely defined by the requirement that, if $P(x,y)=p(x)q(y)$, then
\begin{equation}\label{lambda'=lambdaxlambda}
\lambda^\prime(P)=
\lambda(q)\lambda(p).
\end{equation} 

Moreover,  $\lambda'(u)\in\bD_K(G)$ if and only if $u$ is $\textup{Ad}(K)$-invariant. We set $p^-(x)=p(-x)$.

\begin{corollary} \label{scalarops}
For $D=\lambda'\big(\sum p_j(x)q_j(y)\big)\in \bD_K(G)$, set
$$
A_\tau(D)=\Lambda\Big(\sum_jq_j\otimes d\tau\big(\lambda(p_j^-)\big)\Big)=\sum_j\lambda(q_j)\otimes d\tau\big(\lambda(p_j^-)\big).
$$

Then $A_\tau$ is well defined, it  maps $\bD_K(G)$ linearly onto $\bD(E_\tau)$ and $A_\tau(D)u=Du$ for every $u\in C^\infty_\tau(G,V_\tau)$. 

The identity $A_\tau(D)F=DF$ also holds for $\textup{\rm End}(V_\tau)$-valued smooth functions satisfying the identity $F(xk)=\tau(k)\inv F(x)$,  in particular for $F\in C_{\tau,\tau}\big(G,\textup{\rm End}(V_\tau)\big)$.
\end{corollary}

\begin{proof}
Well defined-ness follows from the linearity of $\lambda$ and $d\tau$; ontoness follows from Theorem \ref{Lambda}.
 
For $u\in C^\infty_\tau(G,V_\tau)$, we have
$$
\begin{aligned}
Du(g)&=\sum_j q_j(\de_y)_{|_{y=0}}p_j(\de_x)_{|_{x=0}}u(g\exp y\exp x)\\
&=\sum_j q_j(\de_y)_{|_{y=0}}p_j(\de_x)_{|_{x=0}}\tau\big(\exp (-x)\big)u(g\exp y)\\
&=\sum_j q_j(\de_y)_{|_{y=0}}d\tau\big(\lambda(p_j^-)\big)u(g\exp y)\\
&=\sum_jd\tau\big(\lambda(p_j^-)\big)\lambda(q_j)u(g)\\
&=A_\tau(D)u(g).
\end{aligned}
$$

The last part of the statement is an obvious consequence of the fact that the space of smooth $\textup{End}(V_\tau)$-valued functions $F$ satisfying the identity $F(xk)=\tau(k)\inv F(x)$ can be identified with $C^\infty_\tau(G,V_\tau)\otimes V'_\tau$.
\end{proof}

\medskip
From the above corollary we can conclude the following :
\begin{itemize}
\item $\textup{Ker}(A_\tau)=
\left\{D\in\mathbb{D}_K(G):D\mid_
{C^\infty_\tau(G,V_\tau)}\right\}
=\left\{D\in\mathbb{D}_K(G):D\mid_
{C^\infty_\tau(G)^{\textup{int}K}}\right\}$ 
\item Let $S_\tau$ and its inverse $S_\tau\inv$ be the operators defined in \eqref{S} and \eqref{Sinv}. For $F\in C^\infty_{\tau,\tau}(G,\textup{End}(V_\tau)$, $f\in C^\infty_\tau(G)^{\textup{int}K}$  and $D\in \bD_K(G)$, Corollary \ref{scalarops} gives the identities
$$
S_\tau\big(A_\tau( D)F\big)=S_\tau(DF)=D(S_\tau F)\ ,\qquad S_\tau\inv(Df)=D(S_\tau\inv f)=A_\tau(D)(S_\tau\inv f).
$$
\end{itemize}

Let $\left(\fS(\frak{p})\otimes\fS(\frak{k})\right)^K$ denotes the Ad$K$ invariant elements in $\fS(\frak{p}\otimes\fS(\frak{k})=\fS(\frak{g})$.

\begin{proposition}\label{diagram}
The diagram

\centerline{
\begin{tikzpicture}[node distance=6.5cm, auto]
\node (A) {$\big(\fS(\frak{p})\otimes \fS(\frak{k})\big)^K$};
\node (B) [right of=A] {$\mathbb{D}_K(G) $};
\node (C) [node distance=3cm, below of=B] {$\mathbb{D}(E_\tau)$};
\node (D) [left of=C] {$\big(\fS(\frak{p})\otimes\textup{End}(V_\tau)\big)^K$};
\draw[->] (A) to node [swap] {$\lambda'$} (B);
\draw[->] (B) to node [swap] {$A_\tau$} (C);
\draw[->] (A) to node [swap] {$I\otimes (d\tau\circ\lambda)$} (D);
\draw[->] (D) to node [swap] {$\Lambda$} (C);
\end{tikzpicture}}
\noindent is commutative, the horizontal arrows indicate bijections and the vertical ones surjections. Moreover, conjugation by $S_\tau$ establishes an isomorphism of algebras between
$\bD(E_\tau)$ and 
$$
\bD_{K,\tau}(G)=\bD_K(G)/\ker A_\tau.
$$
\end{proposition}

\begin{corollary}\label{finite basis}
$\bD(E_\tau)$ and $\bD_{K,\tau}(G)$ are finitely generated algebras. 
\end{corollary}

\begin{proof}
It suffices to prove that $\bD_{K,\tau}(G)$ is finitely generated.
By the Hilbert basis theorem, the space $I(\fg)$ of $\textup{Ad}(K)$-invariant polynomials on $\fg$ is finitely generated. Let $u_1,\dots,u_m$ be a Hilbert basis, with $\deg u_j=d_j$. Then $\la'(u_ju_k)\equiv \la'(u_j)\la'(u_k)$ modulo elements of degree strictly smaller than $d_j+d_k$.  This implies, by induction on the degree, that $\bD_K(G)$ is  generated by $\la'(u_1),\dots,\la'(u_m)$.

Being a quotient of $\bD_K(G)$,  $\bD_{K,\tau}(G)$ is also finitely generated.
\end{proof}

\vskip.5cm

\section{Characterization of spherical functions as eigenfunctions}\label{sect:eigenfunctions}
\vskip.5cm

Like for the standard case of a pair $(G,K)$ \cite{T,W},  commutativity of  convolution algebras is equivalent to commutativity of  algebras of invariant differential operators, under the assumption that $G/K$ is connected. We remark that this statement is usually formulated under the stronger assumption that $G$ is connected.

\begin{theorem}\label{gelfand triple iff diff op commutative III} 
Consider the following statements:
\begin{enumerate}
\item[\rm(i)] $(G,K,\tau)$ is a commutative triple,
\item[\rm(ii)] $\mathbb{D}(E_\tau)$ is commutative,
\item[\rm(iii)] $D_{K,\tau}(G)$ is commutative.
\end{enumerate}

Then {\rm (i)}$\Rightarrow${\rm (ii)}$\Leftrightarrow${\rm (iii)}. If $G/K$ is connected, they are all equivalent.
\end{theorem}

\begin{proof}
The equivalence of (ii) and (iii) follows from Proposition \ref{diagram}.

To prove that (i)$\Rightarrow$(iii), let $D\in D_K(G)$, $f\in C^\infty_\tau(G)^{\textup{int} K}$,
 $u\in C_c^\infty(G)^{\textup{int}K}$. 
 
With $u_\tau=u*(d_\tau\bar\chi_\tau m_K)$,
we have 
\begin{itemize}
\item $Df\in C^\infty_\tau(G)^{\textup{int} K}$,
\item $Du_\tau=(Du)_\tau\in C^\infty_\tau(G)^{\textup{int} K}$,
\item $u*f=u_\tau*f=f*u_\tau=f*u$.
\end{itemize}

Hence, given $D_1,D_2\in D_K(G)$,
$$
\begin{aligned}
u*(D_1D_2f)&=D_1D_2(u*f)=D_1D_2(u_\tau *f)\\
&=D_1D_2(f*u_\tau)=D_1\big(f*(D_2u_\tau)\big)\\
&=D_1\big((D_2u_\tau)*f\big)=(D_2u_\tau)*(D_1f)\\
&=(D_1f)*(D_2u_\tau)=D_2\big((D_1f)*u_\tau\big)\\
&=(D_1f)*(D_2u_\tau)=u_\tau*(D_2D_1f)\\
&=u*(D_2D_1f)\ .
\end{aligned}
$$ 

Applying this identity to an approximate identity $\{u_j\}\subset C_c^\infty(G)^{\textup{int}K}$, we conclude that $D_1D_2f=D_2D_1f$, hence $D_1D_2=D_2D_1$ in the quotient algebra $D_{k,\tau}(G)$.

We consider now the opposite implication (iii)$\Rightarrow$(i) assuming first that $G$ is connected. 

Let  $u$ be any $K$-central analytic function. 
The function 
$$
U(x,y)=\int_Ku_\tau(kxk^{-1}y)dk.
$$ 
 is analytic in $(x,y)\in G\times G$, $K$-central in each variable and satisfies the identity $U(k_1,k_2)=U(k_2,k_1)$ for all $k_1,k_2\in K$. 
 
 Let $D_1,D_2\in \mathbb{D}(G)$. 
 Defining $D_i^0:=\int_KD_i^{\textup{Ad}K}dk$,  for any $K$-central function $g$,   
 $D_i^0g$  is $K$-central.  
Then, for every $k_1,k_2\in K$,
\begin{eqnarray*} 
D_{1,x}D_{2,y}U(k_1,k_2)&=& D_{1,x}^0D_{2,y}^0U(k_1,k_2)\\
&=&\int_K D_{1,x}^0\big(D_2^0u_\tau(kxk^{-1}k_2)\big)\mid_{x=k_1}dk\\
&=&\int_K D_{1,x}^0\big(D_2^0u_\tau(k^{-1}k_2kx)\big)\mid_{x=k_1}dk\\
&=&\int_K D_1^0D_2^0u_\tau(k^{-1}k_2kk_1)\big)dk\\
&=&\int_K D_2^0D_1^0u_\tau(k^{-1}k_2kk_1)\big)dk\\
&=&\int_K D_2^0D_1^0u_\tau(kk_1k^{-1}k_2)\big)dk\\
&=&D_{2,x}D_{1,y}U(k_2,k_1)\\
&=&D_{1,y}D_{2,x}U(k_2,k_1).
\end{eqnarray*} 

Denote by $G_e$ the connected component of $e$ in $G$. 
By the analyticity of $U$, it follows that $U(k_1x,k_2y)=U(k_2y,k_1x)$ for all $k_1,k_2\in K$ and $x,y\in G_e$. Since $G/K$ is connected, every connected component of $G$ contains an element of $K$. Hence $U(x,y)=U(y,x)$ for all $x,y\in G$.

This implies that, for any  $f,g\in L^1_{\tau}(G)^{\textup{int}K}$ and 
compactly supported,
$$
\int_Gg*f(x)u(x)dx=\int_Gg*f(x)u_\tau(x)dx=
\int_G\int_Gf(y)g(x)U(x,y)dxdy$$ and similarly $$\int_Gf*g(x)u(x)d=\int_G\int_Gg(y)f(x)U(x,y)dxdy=
\int_G\int_Gg(x)f(y)U(y,x)dxdy.
$$ 
Therefore $\int_G g*f(x)u(x)dx=\int f*g(x)u(x)$ for all $K$-central analytic functions $u$.
Since analytic integrable functions are dense in $L^1(G)$, the same is true for $L^1(G)^{\textup{int}K}$, hence $f*g=g*f$. 
\end{proof}

\begin{theorem}\label{spherical-diff}
Let $(G,K,\tau)$ be a commutative triple, with $G/K$ connected. For $\Phi\in L^\infty_{\tau,\tau}(G,\textup{End}(V_\tau))$ and $\phi=d_\tau {\rm Tr}\Phi\in L^\infty_\tau(G)^{\textup{int}(K)}$, the following are equivalent : 
\begin{enumerate}
\item[(i)]  $\Phi$ is a $\tau$-spherical function;
\item[(ii)]   $\Phi$ is a joint eigenfunction for all $\textbf{D}\in\big(\mathbb{D}(G)\otimes\textup{End}(V_\tau)\big)^K$  and $\Phi(e)=I$.
\item[(iii)] $\Phi$ is a joint eigenfunction for all $D\in\mathbb{D}_{K}(G)$ and $\Phi(e)=1$.
\item[(iv)] $\phi$ is a trace $\tau$-spherical function;
\item[(v)]  $\phi$ is a joint eigenfunction for all $D\in\mathbb{D}_{K}(G)$ and $\phi(e)=1$.
\end{enumerate}

In particular, both kinds of spherical functions are analytic.
\end{theorem}

\begin{proof}
The equivalence of (ii) and (iii) follows from Corollary \ref{scalarops}.
To complete the proof, it suffices to prove the equivalence of (iv) and (v).

Lifting $K$-central functions on $G$ to bi-$K^\sharp$-invariant functions on $K\ltimes G$, with $K^\sharp={\rm diag}(K)$, one can see that  condition \eqref{functional equation for phi} is equivalent to the functional equation $\int_K\phi^\sharp(gkg')dk=\phi^\sharp(g)\phi^\sharp(g')$ for the lifted function $\phi^\sharp$.

The proof in \cite[Ch. IV, Prop. 2.2]{H} can then be adapted to prove that a non trivial $K$-central bounded function $\phi$ satisfies \eqref{functional equation for phi} if and only if it is an eigenfunction of all $D\in\mathbb{D}_K(G)$. 

Analiticity of $\phi$ follows from the existence of a $K$-central left-invariant Laplacian on $G$, and that of $\Phi$ by Theorem \ref{characterization of spherical function in terms of functional equation-1}.
\end{proof}

\begin{corollary}\label{triple->pair}
Let $(G,K,\tau)$ be a commutative triple with $G/K$ is connected. Then $(G,K)$ is a Gelfand pair. In particular, $G$ is unimodular.
\end{corollary}

\begin{proof}
By Theorem \ref{gelfand triple iff diff op commutative III}, $\mathbb{D}(E_\tau)$ is commutative. Let $I(\frak{p})$ denote the space of all $\textup{Ad}(K)$-invariant polynomials on $\frak{p}$. Then 
clearly $I(\frak{p})\otimes I\subset\big(\fS(\frak{p})\otimes\textup{End}(V_\tau)\big)^K$. Therefore, in view of Theorem \ref{Lambda}, it follows that $\lambda(I(\frak{p}))$ is commutative. But this implies, by \cite[Thm. 4.9]{H}, that $\mathbb{D}(G/K)$ is commutative. Hence $(G,K)$ is a Gelfand pair.
\end{proof}

\vskip.5cm

\section{Gelfand spectrum and embeddings into Euclidean spaces}\label{sect:embeddings}
\vskip.5cm

The following theorem says that the $\tau$-spherical functions are uniquely determined by the set of their eigenvalues.

\begin{theorem}\label{spherical function uniquely determined by eigenvalues}
Let $(G,K,\tau)$ be a commutative triple. Let $\phi_1$ and $\phi_2$ be two trace $\tau$-spherical function so that $D\phi_i=\mu_i(D)\phi_i$ (i=1,2) for all $D\in\mathbb{D}_K(G)$. If $\mu_1(D)=\mu_2(D)$ for all $D\in\mathbb{D}_{K}(G)$, then $\phi_1=\phi_2$. 
\end{theorem}

\begin{proof}
By the given conditions we have, $D\phi_1(e)=D\phi_2(e)$ for all $D\in\mathbb{D}_K(G)$. Now let $D\in\mathbb{D}(G)$. Define $D_0=\int_KD^{\textup{Ad}k}dk$ which clearly belongs to $\mathbb{D}_K(G)$. Since $\phi_i$ are $K$-central, it follows that $D_0\phi_i(e)=D\phi_i(e)$. Therefore $D\phi_1(e)=D\phi_2(e)$ for all $D\in\mathbb{D}(G)$. But $\phi_1,\phi_2$ being analytic, they must coincide on whole of $G$.
\end{proof}

Let $\Sigma_\tau$ be the set of all trace $\tau$-spherical function. 
By Corollary \ref{finite basis}, we can fix  a finite set $\mathcal{D}=(D_1,\cdots,D_k)$ of generators of $\mathbb{D}_{K,\tau}(G)$. Then the map $\rho_{\mathcal{D}}$ which assigns to a $\phi\in\Sigma_\tau$ the $k$-tuple $$\rho_{\mathcal{D}}(\phi)=\big(\lambda_{D_1}(\phi),\cdots,\lambda_{D_k}(\phi)\big)\in\mathbb{C}^k$$  
of its eigenvalues
is injective, by Proposition \ref{spherical function uniquely determined by eigenvalues}. Assume that $\Sigma_\tau$ is endowed with the weak*-topology. 
\begin{theorem}
On $\Sigma_\tau$, the weak*-topology and the compact-open topology coincide. The map $\rho_{\mathcal{D}}$ is a homeomorphism from $\Sigma_\tau$ to its image in $\mathbb{C}^k$, and $\rho_{\mathcal{D}}(\Sigma_\tau)$ is closed.

\end{theorem}

\begin{proof}
Introducing coordinates $t=(t_1,\cdots,t_n)$ on $\frak{g}$ according to the $\exp$ map, let $|t|$ an $\textup{Ad}(K)$-invariant norm on $\frak{g}$. The operator $\triangle:=\lambda(|t|^2)$ is in $\mathbb{D}_K(G)$ and is elliptic (a Laplacian). Being an eigenfunction of $\triangle$, any $\phi\in \Sigma_\tau$ is real-analytic.
The rest of the proof goes as in \cite[Thm. 10 and Cor. 11]{F}. 
\end{proof}

\vskip.5cm

\section{Strong Gelfand pairs}\label{sect:strong}
\vskip.5cm

Given a pair $(G,K)$, with $K$ a compact subgroup of $G$, there may be several $\tau\in\hat K$ such that $(G,K,\tau)$ is a commutative triple. We have already seen that if a nontrivial $\tau$ gives a commutative triple and $G/K$ is connected, then also the trivial representation does.

The extreme situation occurs when $(G,K,\tau)$ is a commutative triple for all $\tau\in\hat K$. When this occurs, one says that $(G,K)$ is a {\it strong Gelfand pair}.

\begin{lemma}\label{strong}
The following are equivalent:
\begin{itemize}
\item[\rm(i)] $(G,K)$ is a strong Gelfand pair,
\item[\rm(ii)] $(K\ltimes G,K)$ is a Gelfand pair, with $K$ acting on $G$ by inner automorphisms,
\item[\rm(iii)] $L^1(G)^{\textup{int}K}$ is commutative,
\item[\rm(iv)] (for $G$  connected) $\mathbb{D}_K(G)$ is commutative.
\end{itemize}
\end{lemma}

This follows easily from the identities
$$
L^1(G)^{\textup{int}K}=\sum_{\tau\in \widehat K}L^1_\tau(G)^{\textup{int}K}\ ,\qquad L^1_\tau(G)^{\textup{int}K}*L^1_\sigma(G)^{\textup{int}K} =\{0\}\ ,\quad \tau\not\sim\sigma\ .
$$

\begin{definition}
Let $(G,K)$ be a strong Gelfand pair. A non zero $\phi\in L^\infty(G)^{\textup{int}K}$ is said to be a $L^1(G)^{\textup{int}K}$-spherical function if the map $$f\rightarrow \int_G f(g)\phi(g^{-1})dg$$ is an an algebra homomorphism of $L^1(G)^{\textup{int}K}$ onto $\mathbb{C}$. 
\end{definition}

The following statement follows directly from (ii) in Lemma \ref{strong}.

\begin{theorem}\label{strong GP equi cond}
Let $(G,K)$ be a strong Gelfand pair. The following are equivalent: 
\begin{itemize}
\item[\rm(i)] $\phi$ is a $L^1(G)^{\textup{int}K}$-spherical function,
\item[\rm(ii)] $\phi$ is nonzero, bounded and satisfies the functional equation 
\begin{eqnarray}\label{strong functional equation}
\int_K\phi(xkyk^{-1})dk=\phi(x)\phi(y),
\end{eqnarray}
\item[\rm(iii)] (if $G$ is connected) $\phi(e)=1$, $\phi$ is $K$-central, analytic and is a joint eigenfunction of all $D\in\mathbb{D}_K(G)$.    
\end{itemize}
\end{theorem}

If $\phi$ is a $L^1(G)^{\textup{int}K}$-spherical function, then, by the functional equation \eqref{strong functional equation}, $\phi(e)=1\neq 0$. Therefore, in view of Theorem \ref{strong GP equi cond} (iii) and Proposition \ref{spherical-diff}, we have the following proposition.

\begin{proposition}
Let $(G,K)$ be a strong Gelfand pair. Let $\phi$ be a $L^1(G)^{\textup{int}K}$-spherical function. Then there is unique $\tau\in\widehat{K}$ such that $\phi\in L^\infty(G)_\tau^{\textup{int}K}$. Hence $\phi$ is an trace $\tau$-spherical function. 
\end{proposition}

\begin{proof}
Since $\phi$ is nonzero, there is $\tau\in \hat K$ such that $\phi_\tau:=\phi*(d_\tau\bar\chi_\tau m_K)$ is nonzero.
By \eqref{strong functional equation}, we obtain that, for all $x,y\in G$,
$$
\phi_\tau(x)\phi(y)=\phi(x)\phi_\tau(y)\ .
$$

With $y=e$, this gives that $\phi_\tau=\phi_\tau(e)\phi$. Hence $\phi_\tau(e)\ne0$, $\phi\in L^\infty(G)_\tau^{\textup{int}K}$, and finally $\phi=\phi_\tau$.
\end{proof}

\begin{corollary}
Let $\Sigma$ denote the spectrum of the Gelfand pair $(K\ltimes G,K)$. Then $\Sigma$ is the disjoint union of the spectra $\Sigma_\tau$ of the triples $(G,K,\tau)$, $\tau\in\hat K$. Each $\Sigma_\tau$ is open and closed in $\Sigma$.
\end{corollary}

\begin{proof}
For each $\tau\in \hat K$, the map $\rho_\tau$ which assigns to a bounded spherical function $\phi\in\Sigma$ the value
$$
\rho_\tau(\phi)=\phi*(d_\tau\bar\chi_\tau m_K)(e)=\int_K\phi(k)\chi_\tau(k)dk\ ,
$$
is continuous on $\Sigma$ and only takes values 0 or 1. This implies that $\Sigma_\tau=\rho_\tau\inv(1)$ is closed. On the other hand, if a sequence of functions $\ph_n\in \Sigma$ converges to $\phi\in\Sigma_\tau$, then $\rho_\tau(\phi_n)$ must be eventually be equal to 1. Then $\phi_n\in\Sigma_\tau$ eventually. This proves that $\Sigma_\tau$ is open.
\end{proof}

\vskip.5cm

\section{$K$-homogeneous bundles over a Lie group $H$}\label{sect:bundles}
\vskip.5cm

In this section we consider the special case where $G=K\ltimes H$, $H$ being  a Lie group and $K$ a compact group of automorphisms of $H$, and $(\tau,V_\tau)$ is an irreducible unitary representation of $K$. We denote by $k\cdot x$ the action of $k\in K$ on $x\in H$. The product on $K\ltimes H$ is given by
$$
(k,x)(k',x')=\big(kk',x(k\cdot x')\big).
$$

The sections of $E_\tau$, i.e., the $V_\tau$-valued functions $u$ on $G$ satisfying the identity $u(gk)=\tau(k)\inv u(g)$, are naturally identified with $V_\tau$-valued functions $u_0$ on $H$, via the map $T$ given by
\begin{equation}\label{fromHtoG}
T:\quad u_0(x)\longmapsto  u(k,x)=\tau(k)\inv u_0(x).
\end{equation}

The action of  $H$ on $u_0$ is given by left translations and that of an element $k\in K$ by
\begin{equation}\label{KonHsections}
k:\quad u_0(x)\longmapsto \tau(k)u_0(k\inv\cdot x).
\end{equation}

Similarly, the integral operators on $V_\tau$-valued functions on $H$ commuting with the action of $G$ are given, in analogy with \eqref{Schwartz-kernel-thm}, by
$$
u(x)\longmapsto \int_HF(y\inv x)u(y)dy,
$$
where $F:H\rightarrow \textup{End}(V_\tau)$ satisfies the identity 
\begin{eqnarray}\label{eqn tau type function} 
F(k\cdot x)=\tau(k)F(x)\tau(k^{-1}).
\end{eqnarray} 

This is equivalent to saying that 
$$
TF(k,x)=\tau(k)\inv F(x) \in L^1_{\tau,\tau}(G).
$$

Composition of integral operators corresponds to convolution of $F_1,F_2\in L^1_\tau(H)$, defined as 
$$
F_1*F_2(x)=\int_H F_2(y^{-1}x)F_1(y)dy.
$$ 
Then, under this convolution, $L^1_\tau(H)$ becomes an algebra.  

Therefore $(K\ltimes H,K,\tau)$ is a commutative triple iff $L^1_\tau(H)$ is commutative. In this case $H$ is unimodular.

As for Gelfand pairs, here too we have a reformulation of Theorem \ref{representation theoretic criteria-general case}. Let $\widehat{H}$ be the dual object of $H$, i.e., the set of equivalence classes $[\pi]$ of irreducible unitary representations $\pi$ of $H$. The group $K$ acts on $\widehat {H}$ in the following way: given $k\in K$ and $\pi$ irreducible and unitary, $\pi^k(x)=\pi(k^{-1}\cdot x)$ defines an irreducible and unitary 
representation of $H$, which may or may not be equivalent to $\pi$. Note that if $\pi_1\sim\pi_2$ then $\pi_1^k\sim\pi_2^k$. So we can set $k\cdot [\pi]=[\pi^k]$. 

Let $K_\pi$ be the stabilizer of $[\pi]$, which is clearly compact. For $k\in K_\pi$, there exists a (unique up to a unitary factor) unitary operator $\delta(k)$ on $\cH_\pi$ (the Hilbert space where $\pi$ is realized) which intertwines $\pi$ with $\pi^k$, i.e., $\pi^k(x)=\delta(k^{-1})\pi(x)\delta(k)$ for all $x\in H$. This defines a projective unitary 
representation $\delta$ of $K_\pi$ on $\cH_\pi$. Consider the representation $\delta\otimes (\tau\mid_{K_\pi})$ of $K_\pi$  on $\cH_\pi\otimes V_\tau$. Since $K_\pi$ is compact, $\delta\otimes(\tau\mid_{K_\pi})$ is completely reducible. Then we have the following theorem giving a characterization for $(G,K,\tau)$ to be a commutative triple, in analogy with \cite{Car}.

\begin{theorem}\label{Gelfand pair iff multiplicity free}
$(K\ltimes H,K,\tau)$ is a commutative triple iff for each $\pi\in\widehat{H}$, $\delta\otimes(\tau\mid_{K_\pi})$ is multiplicity free.
\end{theorem} 

To prove the Theorem we need the following lemma. Let $\mathcal{L}(\cH_\pi)$ denote the set of all bounded operators on 
$\cH_\pi$. 
For $F
\in L^1(H)\otimes \textup{End}(V_\tau)$, 
define $\pi(F)\in\mathcal{L}(\cH_\pi)\otimes \textup{End}(V_\tau)\cong \mathcal{L}(\cH_\pi\otimes V_\tau)$ by 
$$
\pi(F)=\int_H\pi(x)\otimes F(x)dx.
$$ 

\begin{lemma}\label{pi(F) commutes with delta tensor tau}
Let $F\in L^1_\tau(H)$. Then $\pi(F)$ intertwines $\delta\otimes(\tau\mid_{K_\pi})$ with itself.
\end{lemma} 

\begin{proof}
We have, for $k\in K_\pi$,
\begin{eqnarray*}
 \pi(F)\big(\delta(k)\otimes\tau(k)\big)&=&\int_H\big(\pi(x)\delta(k)\big)\otimes\big(F(x)\tau(k)\big)dx\\
 &=&\int_H\big(\delta(k)\pi^k(x)\big)\otimes \big(\tau(k)F(k\inv\cdot x)\big)dx\\
 &=& \int_H\big(\delta(k)\pi(k\inv\cdot x)\big)\otimes \big(\tau(k)F(k\inv\cdot x)\big)dx\\
 &=& \big(\delta(k)\otimes\tau(k)\big) \pi(F).\qquad\qedhere
\end{eqnarray*} 
\end{proof}

Now we prove Theorem \ref{Gelfand pair iff multiplicity free}. For the proof of 'only if' part, we follow \cite{Car}, cf. also  \cite[Thm.3.5]{B} 

\begin{proof}[Proof of Theorem \ref{Gelfand pair iff multiplicity free}]
Let $(G,K,\tau)$ be a commutative triple and let $\pi\in\widehat{H}$. If  $\sigma$ is the multiplier of the associated projective representation $\delta$ of $K_\pi$, let $\widehat{K_\pi}^\sigma$ denote the set of equivalence classes of unitary irreducible projective representation of $K_\pi$ with multiplier $\sigma$. Then 
$$
\delta=\bigoplus_{\rho\in\widehat{K_\pi}^\sigma} c(\rho,\delta)\rho,
$$ 
where $c(\rho,\delta)$ is the multiplicity if $\rho$ in $\delta$. For $\rho\in\widehat{K_\pi}^\sigma$, let $\rho'$ be its contragredient representation, with multiplier $\bar{\sigma}$. Then $R(k,x):=\rho'(k)\otimes \pi(x)\delta(k)$ defines an irreducible linear representation of $K_\pi\ltimes H$ and the induced representation $\textup{Ind}_{K_\pi\ltimes H}^{K\ltimes H}R$ is irreducible. Also, 
$$
(\textup{Ind}_{K_\pi\ltimes H}^{K\ltimes H}R)\mid_K\backsimeq \textup{Ind}_{K_\pi}^{K}(R\mid_{K_\pi})=\textup{Ind}_{K_\pi}^{K}(\rho'\otimes \delta).
$$ 

By Frobenius reciprocity,
$$
c\big(\tau,(\textup{Ind}_{K_\pi\ltimes H}^{K\ltimes H}R)\mid_K\big)
=c(\tau\mid_{K_\pi},\rho'\otimes \delta).
$$

Since $(K\ltimes H,K,\tau)$ is a commutative triple, so is $(K\ltimes H,K,\tau^{\prime})$. Therefore 
$$
c\big(\tau^{\prime},(\textup{Ind}_{K_\pi\ltimes H}^{K\ltimes H}R)\mid_K\big)=0 \textup{ or } 1,
$$ 
and hence $c(\tau^\prime\mid_{K_\pi},\rho'\otimes \delta)=0$ or $1$. But 
$$
c(\tau^\prime\mid_{K_\pi},
\rho'\otimes\delta)=c\big(1,\rho'\otimes \delta\otimes(\tau\mid_{K_\pi})\big)=c\big(\rho,\delta\otimes(\tau\mid_{K_\pi})\big). 
$$ 

Hence it follows that $\delta\otimes(\tau\mid_{K_\pi})$ is multiplicity free.

Conversely, assume that the $K_\pi$-action on $\cH_\pi\otimes V_\tau$ is multiplicity free for all $\pi\in\widehat{H}$. Then it follows from Lemma \ref{pi(F) commutes with delta tensor tau} that $\pi(F)$ and $\pi(G)$ commutes whenever $F,G\in L^1_\tau(H)$. Since this is true for all $\pi\in\widehat{H}$, by uniqueness of the Fourier transform we can conclude that $F*G=G*F$.   
\end{proof}

\begin{definition}
Let $(K\ltimes H,K,\tau)$ be a commutative triple. A non-trivial function $\Psi \in L^\infty_\tau(H)$ is said to be a $\tau$-spherical function if the map $$F\rightarrow \widehat F (\Psi ):=\frac{1}{d_\tau}\int_H \textup{Tr}[\Psi (x^{-1})F(x)]dx$$ is a homomorphism of $L^1_\tau(H)$ into $\mathbb C$. Here $d_\tau$ is the dimension of $V_\tau$.
\end{definition} 

Recall the definition of $T$ from \ref{fromHtoG}.

\begin{theorem}
The following are equivalent:
\begin{enumerate}
\item[\rm(i)] $\Psi $ is a $\tau$-spherical function on $H$;
\item[\rm(ii)] $\Phi=T(\Psi )$ is a $\tau$-spherical function on $K\ltimes H$;
\item[\rm(iii)] $\Psi \in L^\infty_\tau(H)$ is non-trivial and, for all $h,h'\in H$, satisfies the identity
\begin{eqnarray}\label{functional equation for Phi_0}
d_\tau\int_K\tau(k^{-1})\Psi \big(h(k\cdot h^\prime)\big)\chi_\tau(k)dk=\Psi (h^\prime)\Psi (h).
\end{eqnarray}
\end{enumerate}
\end{theorem}

\begin{proof}
An easy calculation shows that $\widehat{T(F)}(\Phi)=\widehat{F}(\Psi )$ for all $F\in L^1_\tau(H)$. This proves the equivalence of (i) and (ii).

In view of Theorem \ref{characterization of spherical function in terms of functional equation-1}, it is enough to show that \eqref{functional equation for Phi} is equivalent to \eqref{functional equation for Phi_0}. If $x=(e,h)$ and $y=(e,h^\prime)$ then $\Phi(xky)=\Phi(k,hk\cdot h^\prime)=\tau(k^{-1})\Psi (hk\cdot h^\prime)$. Therefore, putting $x=(e,h)$ and $y=(e,h^\prime)$ in \eqref{functional equation for Phi} we get \eqref{functional equation for Phi_0}. Conversely, if $x=(k_1,h_1)$ and $y(k_2,h_2)$ then 
\begin{eqnarray*} 
d_\tau\int_K\Phi(xky)\chi_\tau(k)dk 
&=& d_\tau\int_K\Phi\big(k_1kk_2,h_1(k_1k\cdot h_2)\big)\chi_\tau(k)dk\\
&=&d_\tau\int_K \tau(k_1kk_2)^{-1}\Psi \big(h_1(k_1k\cdot h_2)\big)\chi_\tau(k)dk\\
&=&d_\tau\int_K \tau(k_1kk_2)^{-1}\tau(k_1)\Psi \big((k_1^{-1}\cdot h_1)(k\cdot h_2)\big)\tau(k_1^{-1})\chi_\tau(k)dk\\
&=&\tau(k_2^{-1})\,d_\tau\int_K \tau(k^{-1})\Psi \big((k_1^{-1}\cdot h_1)(k\cdot h_2)\big)\chi_\tau(k)dk\,\tau(k_1^{-1})\\
&=&\tau(k_2^{-1})\Psi (h_2)\Psi (k_1^{-1}\cdot h_1)\tau(k_1^{-1})=\Phi(y)\Phi(x).\qquad\qedhere
\end{eqnarray*}  
\end{proof}

\vskip.5cm

\section{Differential operators on $K$-homogeneous bundles over $H$}\label{sect:diffonH}
\vskip.5cm

Once the sections of $E_\tau$ have been identified with $V_\tau$-valued functions on $H$, we can also realize the elements of $\bD(E_\tau)$ as differential operators on $V_\tau$-valued functions on $H$ which are left-invariant and commute with the action \eqref{KonHsections} of $K$. 

We denote the algebra of such operators by 
$\big(\bD(H)\otimes \textup{End}(V_\tau)\big)^K$.
\smallskip

Keeping Theorem \ref{Lambda} in mind, we can choose $\fp$, the $\textup{Ad}(K)$-invariant complement of $\fk$ in $\fg$, to be $\fh$, the Lie algebra of $H$, and define the map
$$
\Lambda'=\lambda_\fh\otimes I:\big(\fS(\fh)\otimes\textup{End}(V_\tau)\big)^K\longrightarrow \big(\bD(H)\otimes \textup{End}(V_\tau)\big)^K.
$$

The two maps $\Lambda$ of Lemma \ref{Symm} and $\Lambda'$ are conjugate of each other under the map $T$ in \eqref{fromHtoG}. This gives the following theorem.

\begin{theorem}\label{relation between diff operator iii and iv}
The algebras $\bD(E_\tau)$ and $\big(\mathbb{D}(H)\otimes\textup{End}(V_\tau)\big)^K$ are isomorphic. 
In particular, $(K\ltimes H,K,\tau)$ is a commutative triple if and only if $\big(\mathbb{D}(H)\otimes \textup{End}(V_\tau)\big)^K$ is commutative. 
\end{theorem}

\begin{corollary}
Let $(K\ltimes H,K,\tau)$ be a commutative triple and $\Psi \in L^\infty_{\tau}(H)$. The following are equivalent: 
\begin{itemize}
\item[\bf{\rm(i)}] $\Psi $ is a $ \tau$-spherical function.
\item[\bf{\rm(ii)}] $\Psi $ is a joint eigenfunction for all $\textbf{D}\in\big(\mathbb{D}(H)\otimes\textup{End}(V_\tau)\big)^K$.  
\end{itemize}
\end{corollary}

It must be noticed that we have no analogue of Corollary \ref{scalarops} for differential operators on $H$, i.e., the effective action of $\big(\mathbb{D}(H)\otimes \textup{End}(V_\tau)\big)^K$ on $C^\infty$ $V_\tau$-valued functions includes the action of $\mathbb{D}_K(H)$ properly. 
However, we have the following proposition.

\begin{proposition}\label{finitely generated}
$\big(\mathbb{D}(H)\otimes\textup{End}(V_\tau)\big)^K$ is a finite module over  $\mathbb{D}_K(H)$.
\end{proposition}

\begin{proof}
Given $P\in \big(\fS(\fh)\otimes \textup{End}(V_\tau)\big)^K$, consider the characteristic polynomial
$$
\det\big(P(x)-\lambda I\big)=\pm\big(\lambda^d+q_1(x)\lambda^{d-1}+\cdots+q_d(x)\big)\ ,
$$
where $q_j\in \fS(\fh)^K$ for $j=1,\dots,d$. It follows from the Cayley-Hamilton theorem that
$$
P(x)^d=-q_1(x)P(x)^{d-1}-\cdots-q_d(x)I\ .
$$

Combining this with Corollary \ref{finite basis}, we obtain the conclusion.
\end{proof}

\vskip.5cm

\section{Relations among the spectra $\Sigma_\tau$ and $\Sigma_1$}\label{sect:spectra}

\vskip.5cm

It follows from Corollary \ref{triple->pair} that, if $(K\ltimes H,K,\tau)$ is a commutative triple, the same is true with $\tau$ replaced by the trivial representation, i.e., $(K\ltimes H,K)$ is a Gelfand pair. We establish relations between the two spectra, $\Sigma_\tau$ and $\Sigma_1$ respectively, also in terms of their embeddings into complex spaces as introduced in Section~\ref{sect:embeddings}. We start from the following statement.

\begin{lemma}
Given an $\tau$-spherical function $\Psi $, the function 
$\psi=d_\tau\inv\textup{Tr}\,\Psi $
is a usual spherical function for the Gelfand pair $(K\ltimes H,K)$. For $D\in\bD(H)^K$, the following relation between eigenvalues holds:
$$
\la_D(\psi)=\la_{D\otimes I}(\Psi )\ .
$$

The map $d_\tau\inv\textup{Tr}$ is continuous from $\Sigma_\tau$ to $\Sigma_1$.
\end{lemma}

\begin{proof}
Since $\Psi $ satisfies \eqref{eqn tau type function}, $\textup{Tr}\,\Psi $ is $K$-invariant.
Moreover, $\psi(e)=1$ and, for $D\in\bD(H)^K$,
$$
D\psi=\frac1{d_\tau}D(\textup{Tr}\,\Psi )=\frac1{d_\tau}\textup{Tr}\,\big((D\otimes I)\Psi \big)=\la_{D\otimes I}(\Psi )\psi.
$$

Continuity of the map $\Psi \mapsto\psi$ with respect to the topologies of uniform convergence on compact sets is obvious.
\end{proof}

Reformulating  the results of Section \ref{sect:eigenfunctions} in terms of $\textup{End}(V_\tau)$-valued spherical functions and differential operators, and specializing to the present situation where $G=K\ltimes H$, such embeddings depend on the choice of a finite generating system  $\cD=\{\textit{\textbf D}_j\}\subset\big(\mathbb{D}(H)\otimes\textup{End}(V_\tau)\big)^K$.

It is convenient to consider systems $\cD$ formed by a generating system $\cD_0=\{D_1,\dots, D_h\}$ of $\bD_K(H)$, completed with  $\textit{\textbf D}_{h+1},\dots, \textit{\textbf D}_d$ generating $\big(\mathbb{D}(H)\otimes\textup{End}(V_\tau)\big)^K$ as a $\bD_K(H)$-module.
Then every element $\textit{\textbf D}\in\big(\mathbb{D}(H)\otimes\textup{End}(V_\tau)\big)^K$ can be expressed in the form
$$
\textit{\textbf D}=L_0\otimes I+\sum_{j=h+1}^d L_j\textit{\textbf D}_j,
$$
with $L_0,L_{h+1},\dots,L_d\in \bD(H)^K$, it follows that, setting $\textit{\textbf D}_j=D_j\otimes I$ for $j=1,\dots,h$, $\big(\mathbb{D}(H)\otimes\textup{End}(V_\tau)\big)^K$ is generated by 
$$
\cD=\{\textit{\textbf D}_1,\dots, \textit{\textbf D}_d\}.
$$

Simultaneously, the generating system $\cD_0=\{D_1,\dots, D_h\}$ of $\bD_K(G)$ induces an embedding $\rho_{\cD_0}$ of the spectrum $\Sigma_1$ of the Gelfand pair $(G,K)$ into $\bC^h$.

With $\pi_1$ denoting the canonical projection of $\bC^h\times\bC^{d-h}$ onto its first factor, we have the following commutative diagram:

\centerline{
\begin{tikzpicture}[node distance=6.5cm, auto]
\node (A) {$\Sigma_\tau$};
\node (B) [node distance=7cm, right of=A] {$\rho_\cD(\Sigma_\tau)$};
\node (C) [node distance=3cm, below of=B] {$\rho_{\cD_0}(\Sigma_1)$};
\node (D) [node distance=7cm, left of=C] {$\Sigma_1$};
\draw[->] (A) to node [swap] {$\rho_\cD$} (B);
\draw[->] (B) to node [swap] {$\pi_1$} (C);
\draw[->] (A) to node [swap] {$d_\tau\inv\textup{Tr}$} (D);
\draw[->] (D) to node [swap] {$\rho_{\cD_0}$} (C);
\end{tikzpicture}
}

In particular, $\pi_1$ maps $\rho_\cD(\Sigma_\tau)$ into $\rho_{\cD_0}(\Sigma_1)$ and, if $\psi=d_\tau\inv\textup{Tr}\,\Psi $, $\rho_{\cD_0}(\psi)=\pi_1\circ\rho_\cD(\Psi )$.

\vskip.5cm

\section{Spherical functions of positive type}\label{sect:positive}
\vskip.5cm

Let $V$ be a finite dimensional Euclidean complex space, i.e., endowed with a positive definite Hermitean product.
A continuous ${\rm End}(V)$-valued function $F$ is of positive type on a group $G$ if any of the following equivalent conditions holds, for every finite choice of elements $x_1,\dots,x_n\in G$:
 \begin{enumerate}
 \item for every $v_1,\dots,v_n\in V$, 

 \begin{eqnarray}\label{positive definite function}
\sum_{j,k}\langle F( x_j x_k^{-1})v_k,v_j\rangle\geq 0. 
\end{eqnarray}

\item for every $v_1,\dots,v_n\in V$, the matrix 
$$
\big(\langle F(x_jx_k\inv)v_k,v_j\rangle\big)_{jk}
$$ 
is positive semi-definite;
\item for every  $B_1,\dots,B_n\in {\rm End}(V)$, 
$$
\sum_{j,k}B_j^*F(x_jx_k\inv)B_k\in {\rm End}(V)
$$ 
is positive semi-definite;
\item for any $E\in C_c^\infty(G,\textup{End}(V)))$ 
\begin{eqnarray}\label{equivalent condition for positive definite}
\int_G\textup{Tr}\big[(E^**E)(x)F(x^{-1})\big]dx\geq 0.
\end{eqnarray}

\end{enumerate}

Moreover any measurable function $F$ satisfying condition (4) is continuous and of positive type.

\begin{proposition}
Let $F$ be of positive type. The following properties hold.
\begin{enumerate}
\item[(i)] $F(e)$ is positive semi-definite;
\item[(ii)] for every $x\in G$, $F(x^{-1})=F(x)^*$;
\item [(iii)] fore every $v\in V$, $\psi_v(x)=\langle F(x)v,v\rangle$ is of positive type;
\item[(iv)] $\ker F(e)\subseteq \ker F(x)$ and ${\rm im\,}F(x)\subseteq {\rm im\,}F(e)$  for every $x\in G$;
\item[(v)] for every $x\in G$, $F(x)^*F(x)\le F(e)^2$.
\end{enumerate}
\end{proposition}

\begin{proof}
To prove (i), apply \eqref{positive definite function} with $n=1$. 

To prove (ii), apply \eqref{positive definite function} with $n=2$, $x_1=x$, $x_2=e$. Then, for every $v_1,v_2$,
\begin{equation}\label{2}
\langle F(e)v_1,v_1\rangle+\langle F(x)v_1,v_2\rangle+ \langle F(x^{-1})v_2,v_1\rangle+ \langle F(e)v_2,v_2\rangle\ge0\ .
\end{equation}
Then
$$
\textup{Im} \big(\langle F(x)v_1,v_2\rangle+ \langle F(x^{-1})v_2,v_1\rangle\big)=0\ ,
$$
i.e., $\textup{Im} \langle F(x)v_1,v_2\rangle=\textup{Im} \langle F(x^{-1})^*v_1,v_2\rangle$. Replacing $v_1$ by $iv_1$, we find that also the real parts are equal, and (ii) follows.

(iii) follows from condition (2) with $v_j=v$ for all $j$.

To prove (iv), apply \eqref{positive definite function} with $n=2,x_1=x,x_2=e,v_1=v_1,v_2=\lambda v_2$, where $\lambda$ is real. Then we get $$\langle F(e)v_1,v_1\rangle+\lambda^2\langle F(e)v_2,v_2\rangle+ \lambda\langle F(x^{-1})v_2,v_1\rangle+ \lambda\langle F(x)v_1,v_2\rangle\ge0\ .$$ If $v_1\in\textup{Ker}F(e)$, $$\lambda^2\langle F(e)v_2,v_2\rangle+
2\lambda\textup{Re}\langle
F(x)v_1,v_2\rangle\geq 0$$ which implies that $$|\lambda|\langle F(e)v_2,v_2\rangle\geq |\textup{Re}\langle F(x)v_1,v_2\rangle|$$ for all real $\lambda$.
Taking $\lambda\rightarrow 0$, we deduce that $\textup{Re}\langle F(x)v_1,v_2\rangle=0$. Replacing $v_2=iv_2$, we also get that imaginary part of $\langle F(x)v_1,v_2\rangle$ is zero, so that $\langle F(x)v_1,v_2\rangle=0$. Since this is true for any $v_1\in\textup{Ker} F(e),v_2\in V$, we conclude that $\textup{Ker} F(x)\subset\textup{Ker} F(e).$ 

(v) is equivalent to the condition $\| F(x)v\|\le\| F(e)v\|$ for all $v\in V$. By (iv), it is sufficient to take $v\in\big(\textup{ker} F(e)\big)^\bot$. 
So we may assume that $ F(e)$ is invertible. In this case, replacing $ F(x)$ by $ F(e)^{-\frac{1}{2}} F(x) F(e)^{-\frac{1}{2}}$, which remains of positive type, we may assume that $ F(e)=I$.
Applying \eqref{2} with $v_2$ replaced by $e^{i\theta} F(x)v_2$, we have
$$
\|v_1\|^2+2\textup{Re}\big(e^{-i\theta}\langle F(x)v_1, F(x)v_2\rangle\big)+\| F(x)v_2\|^2\ge0\ ,
$$
which gives, by the arbitrarity of $\theta$,
$$
2\big|\langle v_1, F(x)^* F(x)v_2\rangle\big|\le \|v_1\|^2+\| F(x)v_2\|^2\ .
$$
Passing to the supremum over $v_1$ of unit norm, we obtain the inequality
$$
2\| F(x)^* F(x)v_2\|\le 1+\| F(x)v_2\|^2\ .
$$
For $\|v_2\|=1$, this imples that
$$
2\| F(x)v_2\|^2=2\langle  F(x)^* F(x)v_2,v_2\rangle\le 2\| F(x)^* F(x)v_2\|\le 1+\| F(x)v_2\|^2\ ,
$$
whence $\| F(x)v_2\|\le1$.
\end{proof}

The proofs of the following result can be found in \cite{S}, \cite[vol. II p. 15-16 and Remark]{War}.

\begin{theorem}\label{positive-representation}
Let $(G,K,\tau)$ be a commutative triple. Given $\pi\in\widehat G$ such that $\tau\subset \pi_{|_K}$, (say, with $V_\tau \subset \cH_\pi$), the function $\Phi$ given by
\begin{equation}\label{Phi-postype}
\langle\Phi(x)u,v\rangle=\langle\pi(x)u,v\rangle\ ,
\end{equation}
 with $u,v\in V_\tau$, is $\tau$-spherical and of positive type.
Conversely, every spherical function of positive type arises in this way.
\end{theorem}

\medskip
Consider now the case where $G=K\ltimes H$.

\begin{proposition}\label{equivalent condition for positive definite proposition}
Let $F$ be an $\textup{End}(V_\tau)$-valued measurable function on $H$. Then $F$ is of positive type if and only if $T(F)$, defined according to \eqref{fromHtoG}, is of positive type on $G$.
\end{proposition}

This leads us to the following description of the $L^1_\tau(H)$-spherical functions of positive type only in terms of irreducible unitary representations of $H$, instead of representations of $K\ltimes H$. 

Fix  $\pi\in\widehat{H}$. Let $\cH_\pi\otimes V_\tau=\oplus_{\alpha} W_\alpha(\pi)$ be the multiplicity-free decomposition into irreducible invariant subspaces under the action of $\delta\otimes \tau$, where $\alpha$ runs over an index set $\Lambda=\Lambda(\pi).$ Let $P_\alpha=P_\alpha(\pi)$ denote the orthogonal projection onto $W_\alpha=W_\alpha(\pi)$. If $F\in L^1_\tau(H)$, $\pi(F)\in\mathcal{L}_{K_\pi}(\cH_\pi\otimes V_\tau)$ by Lemma \ref{pi(F) commutes with delta tensor tau}. Therefore $\pi(F)=\oplus_{\alpha}\widehat{F}(\pi,\alpha)P_\alpha$ for some constants $\widehat{F}(\pi,\alpha)$. Since $\pi(F*G)=\pi(F)\pi(G)$, it follows that, for each $\alpha$, $F\rightarrow\widehat{F}(\pi,\alpha)$ defines a multiplicative linear functional of $L^1_\tau(H)$. 

Since for $F\in L^1_\tau(H)$, $\pi(F)$ can be written as 
$$
\pi(F)=\int_K\int_H\big [I\otimes \tau(k^{-1})\big]\big[\pi(k^{-1}\cdot x)\otimes F(x)\big][I\otimes\tau(k)]dxdk,
$$ 
we have 
\begin{eqnarray*}
\widehat{F}(\pi,\alpha)&=&\frac{1}{d_\alpha}\textup{Tr}
[\pi(F)P_\alpha]\\&=&\frac{1}{d_\alpha}\textup{Tr}\bigg[\int_K\int_H\big[\pi(k^{-1}\cdot x)\otimes F(x)\big][I\otimes\tau(k)]P_\alpha\big [I\otimes \tau(k^{-1})\big]dxdk\bigg],
\end{eqnarray*} 

Defining $\textup{Tr}_{\cH_\pi}$ as the partial trace of an element of $\cL(\cH_\pi\otimes V_\tau)$ relative to $\cH_\pi$, and setting
\begin{eqnarray}\label{formula of tau spherical function general case}
\Phi_{\pi,\alpha}(x)&=&\frac{d_\tau}{d_\alpha}\textup{Tr}_{\cH_\pi}\bigg[\int_K\big[\pi(k^{-1}\cdot x)\otimes I\big][I\otimes\tau(k)]P_\alpha\big [I\otimes \tau(k^{-1})\big]dk\bigg],
\end{eqnarray} 
we then have 
\begin{eqnarray*}
\widehat{F}(\pi,\alpha)&=&\frac{1}{d_\alpha}\textup{Tr}\bigg[\int_H\big[I\otimes F(x)\big]\int_K\big[\pi(k^{-1}\cdot x)\otimes I\big][I\otimes\tau(k)]P_\alpha\big [I\otimes \tau(k^{-1})\big]dkdx\bigg]\\&=&\frac{1}{d_\tau} \textup{Tr}\bigg[\int_H F(x)\Phi_{\pi,\alpha}(x)dx\bigg].
\end{eqnarray*} 

This is true for all $F\in L^1_\tau(H)$. Also, note that $\Phi_{\pi,\alpha}\in L^{\infty}_\tau(H)$. Therefore $\Phi_{\pi,\alpha}$ is a $\tau$-spherical function.

In general, not all bounded spherical functions are of positive type. The following case where the two classes coincide is particularly relevant in view of Vinberg's structure theorem for Gelfand pairs \cite[Thm.~5]{V}. The method of proof is taken from \cite{B} and adapted to the nonscalar case.

\begin{theorem}\label{positive on nilpotent}
Let $(K\ltimes H,K,\tau)$ be a commutative triple, with $H$  nilpotent. Then all bounded $\tau$ -spherical functions on $H$ are  of positive type.
\end{theorem}
\begin{proof}
Let $\Psi $ be a bounded $\tau$-spherical function on $H$ i.e. the linear map $\lambda_{\Psi }: F\rightarrow\widehat{F}(\Psi )$ is a non-trivial homomorphism of $L^1_\tau(H)$ onto $\mathbb{C}$. Since $H$ is a locally compact nilpotent lie group, by Corollary 6 in \cite{P}, $L^1(H,\textup{End}(V_\tau))$ is symmetric Banach $*$-algebra. By Theorem 1 in \cite{N} (page-305), it follows that, there is an 
irreducible representation $(\pi,\cH_\pi)$ of $L^1(H,\textup{End}
(V_\tau))$ together with a one dimensional subspace $H_{\Psi }$ of 
$\cH_\pi$ such that 
$\big(\pi\mid_{L^1_\tau(H)},H_{\Psi }\big)$ 
is equivalent to $(\lambda_{\Psi },\mathbb{C})$. Therefore, there is 
a vector $v_0\in H_{\Psi }$ such that $\widehat{F}(\Psi )=\langle 
\pi(F)v_0,v_0\rangle$ for all $F\in L^1_\tau(H
)$. Now, we use Proposition \ref{equivalent condition for 
positive definite proposition} to prove that $\Psi $ is  of positive type. Let $G\in C^\infty_c(H,\textup{End}(V_\tau))$. Define 
$$G^\#(x):=\int_K\tau(k^{-1})F(k\cdot x)\tau(k)dk.$$ Then 
\begin{eqnarray*} 
\int_H\textup{Tr}\big[(G^**G)(x)\Psi (x^{-1})\big]dx &=&
\int_H\textup{Tr}\big[(G^**G)^\#(x)\Psi (x^{-1})\big]dx\\ &=&
\int_H\textup{Tr}\big[((G^\#)^**G^{\#})(x)\Psi (x^{-1})\big]dx\\&=& d_\tau\big\langle\pi((G^\#)^**G^{\#})v_0,v_0\big\rangle\\&=& d_\tau\big\langle\pi(G^{\#})v_0, \pi(G^\#)v_0\big\rangle\geq 0.
\end{eqnarray*} 

Therefore $\Psi $ is  of positive type.  
\end{proof}

\vskip.5cm

\section{The case where $H=\bR^n$ or  the Heisenberg group $H_n$. Commutativity criteria}\label{sect:H=Rn,Hn}
\vskip.5cm

Let $K$ be a compact subgroup of $O(n)$ acting naturally on $\mathbb{R}^n$, and  $(\tau,V_\tau)$ be an irreducible unitary representation of $K$. 

The following theorem gives a criterion for $(\mathbb{R}^n,K,\tau)$ to be a commutative triple. 
\begin{theorem}\label{eqn R^n gelfand pair}
Let $K_x=\{k\in K:k\cdot x=x\}$ be the stabilizer of $x\in\mathbb{R}^n$. Then $(K\ltimes\mathbb{R}^n,K,\tau)$ is a commutative triple if and only if, for all $x\in \mathbb{R}^n$, $K_x$ action on $V_\tau$ is multiplicity free. 
\end{theorem}

The condition ``for all $x\in \mathbb{R}^n$'' can be replaced by ``for generic $x\in \mathbb{R}^n$''.
Theorem \ref{eqn R^n gelfand pair} can be shown to be a consequence of Theorem \ref{Gelfand pair iff multiplicity free}, but it admits a direct  proof, based on following  lemma.
We denote by $\textup{End}_{K_{x_0}}(V_\tau)$ the elements of $\textup{End}(V_\tau)$ which commute with $\tau_{|_{K_{x_0}}}$.

\begin{lemma}\label{existance of tau type function}
Let $x_0\in\mathbb{R}^n$ and $A\in\textup{End}_{K_{x_0}}(V_\tau)$. Then there is a compactly supported function $F\in C_\tau(\mathbb{R}^n)$ such that $F(x_0)=A$.
\end{lemma}

\begin{proof}
Let $S_{x_0}$ be a slice at $x_0$. For the existence of slices, cf. \cite[Ch.~2,\,Section~5]{bre}.
 By definition, $S_{x_0}$ is an open and $K_{x_0}$-invariant neighborhood of $0$ in  the normal space $N_{x_0}$ to the orbit $K\cdot x_0$ at $x_0$, such that the $K$-equivariant map 
 $$
 \sigma:K\times_{K_{x_0}}S_{x_0}\rightarrow \mathbb{R}^n,
 $$ 
 given by $\sigma(k,x)=k(x_0+x)$, is a diffeomorphism of $K\times_{K{x_0}}S_{x_0}$ onto the open neighborhood $K(x_0+S_{x_0})$ of $K\cdot{x_0}$. 
Here the notation $K\times_{K_{x_0}}S_{x_0}$ stands for the quotient of $K\times S_{x_0}$ modulo the action of $K_{x_0}$, i.e., $(kk^\prime,x)$ is equivalent to $(k,k^\prime x)$ for all $k^\prime\in K_{x_0}, x\in S_{x_0}$. Also, cf. \cite[Corollary 5.3]{FRY}, for every $x\in S_{x_0}$, we have the inclusion $K_x\subset K_{x_0}$, more explicitly $K_x=(K_{x_0})_x$.

Fix a $K_{x_0}$-invariant scalar-valued function $\psi\in C_c(S_{x_0})$  with $\psi(x_0)=1$, and define 
$$
F:K\times_{K_{x_0}}S_{x_0}\rightarrow\textup{End}(V_\tau)
$$ 
by 
$$
F(k,x)=\psi(x)\tau(k)A\tau(k^{-1}).
$$ 

This is well defined because, for $k\in K,k^{\prime}\in K_{x_0},x\in S_{x_0}$, $F(kk^{\prime},x)=F(k,k^\prime x)$, since $A\in\textup{End}_{K_{x_0}}(V_\tau)$ and $\psi$ is $K_{x_0}$-invariant. Clearly, $F$ is continuous. Also note that $F(k_1k_2,x)=\tau(k_1)F(k_2,x)\tau(k_1)^{-1}$. 
\end{proof}

\begin{proof}[Proof of Theorem \ref{eqn R^n gelfand pair}]
For $F\in L^1(\bR^n)\otimes\textup{End}(V_\tau)$, define its Fourier transform $\cF F=\widehat F\in C_0(\bR^n)\otimes\textup{End}(V_\tau)$ via the ordinary Fourier integral. Then 
$$
\cF:L^1_{\tau}(\mathbb{R}^n)\longrightarrow \big(C_0(\bR^n)\otimes\textup{End}(V_\tau)\big)^K:=(C_0)_{\tau}(\mathbb{R}^n)\ ,
$$
with dense image. Moreover, $\widehat{F_1*F_2}=\widehat{F}_1\widehat{F}_2$. So, it is enough to show that $(C_0)_{\tau}(\mathbb{R}^n)$ is commutative under pointwise product if and only if, for each $x\in\mathbb{R}^n$, $K_x$ acts on $V_\tau$ without multiplicities.

Assume that, for each $x$, $K_x$ acts on $V_\tau$ without multiplicities. Then $\textup{End}_{K_x}(V_\tau)$ is commutative. If $F\in (C_0)_{\tau}(\mathbb{R}^n)$, $F(k\cdot x)=\tau(k)F(x)\tau(k^{-1})$ for all $k\in K$, which implies that $F(x)\in \textup{End}_{K_x}(V_\tau)$. Hence $(C_0)_\tau(\mathbb{R}^n)$ is commutative.

Conversely, suppose that $(C_0)_\tau(\mathbb{R}^n)$ is commutative. Given $x\in\bR^n$ and $A,B\in \textup{End}_{K_x}(V_\tau)$, by Lemma \ref{existance of tau type function} there exist $F,G\in (C_c)_\tau(\bR^n)$ with $F(x)=A$, $G(x)=B$. This implies that $AB=BA$.
\end{proof}

\medskip

Let now $H_n$ be the $(2n+1)$ dimensional Heisenberg group, identified with $\bC^n\times\bR$ with product
$$
(z,t)(z',t')=\Big(z+z',t+t'+\half\lan z,z'\ran\Big)\ ,
$$
where $\lan z,z'\ran$ denotes the natural Hermitian product on $\bC^n$. For $k\in U(n)$, 
$$
k\cdot(z,t)=(kz,t)
$$
is an automorphism of $H_n$, and, up to conjugation,  $U(n)$ is the unique maximal compact and connected subgroup of   $\textup{Aut}(H_n)$.

Given a compact subgroup $K$ of $U(n)$ and $\tau\in \widehat K$, we have the following necessary and sufficient condition  for $(K\ltimes H_n,K,\tau)$ to be a commutative triple. By $\cP_{\rm hol}(\bC^n)=\bC[z_1,\dots,z_n]$ we denote the space of holomorphic polynomials on $\bC^n$ and by $\sigma$ be the representation of $U(n)$ on $\cP_{\rm hol}(\bC^n)$ given by $\sigma(k)P=P\circ k\inv$.

\begin{proposition}\label{multfreeH_n}
Given a compact subgroup $K$ of $U(n)$ and $\tau\in \widehat K$, the triple $(K\ltimes H_n,K,\tau)$ is commutative if and only if $\sigma_{|_K}\otimes \tau$ and $\sigma'_{|_K}\otimes\tau$ decompose without multiplicities.
\end{proposition}

\begin{proof}
In applying Theorem \ref{Gelfand pair iff multiplicity free} to our case, it suffices to restrict oneself to the infinite-dimensional representations $\pi_\la\in\widehat H$ associated to nontrivial characters $e^{i\la t}$ of the center of $H_n$. 

In the Fock-Bargmann model, the representation space consists of holomorphic functions on $\bC^n$ which are square integrable with respect to the measure $e^{-|\lambda||z|^2}dzd\bar{z}$; and polynomials are dense in it. Moreover, the stabilizer $K_{\pi_\la}$ is the full group $K$ and the representation $\delta$ described in the proof of Theorem \ref{Gelfand pair iff multiplicity free} is linear and coincides with $\sigma_{|_K}$ for $\la>0$ and with $\sigma'_{|_K}$ for $\la<0$.

The conclusion is then immediate.
\end{proof}

\begin{corollary}
The triples $\big(SO(n)\ltimes\bR^n,SO(n),\tau\big)$, $\big(U(n)\ltimes H_n,U(n),\tau\big)$ are commutative for every $\tau\in\widehat{K}$.

For the triples $\big(SU(n)\ltimes\bC^n,SU(n),\tau\big)$,  $\big(SU(n)\ltimes H_n,SU(n),\tau\big)$ the following  are equivalent:
\begin{enumerate}
\item[\bf(i)] it is commutative,
\item[\bf(ii)] the highest weight $\mu$ of $\tau\in\widehat{SU(n)}$ is singular, i.e. is annihilated by some root,
\item[\bf(iii)] condition \eqref{eqn tau type function} on an integrable $\textup{End}(V_\tau)$-valued function $F$ implies the identity $F(e^{i\theta} z)=F(z)$, resp. $F(e^{i\theta} z,t)=F(z,t)$, i.e., the $SU(n)$-action on the bundle extends to $U(n)$.
\end{enumerate}
\end{corollary}

\begin{proof}
The first part is a restatement of the well-known fact that $\big(SO(n)\ltimes\bR^n,SO(n)\big)$, $\big(U(n)\ltimes H_n,U(n)\big)$ are strong Gelfand pairs \cite[Thm. 3]{Y}.  On the other hand, the 
first case follows directly from Theorem \ref{eqn R^n gelfand pair} applying the branching rules for restrictions of representations of classical groups \cite{FH}, and the second can be proved by  modifying slightly  the argument below.

To prove the part of the statement  concerning $H=H_n$ and $K=SU(n)$,  we disregard the trivial case $n=1$ and assume that $n\ge2$.
We have
\begin{equation}\label{m-decomposition}
\sigma_{|_{SU(n)}}\otimes \tau=\sum_{m=0}^\infty \sigma_m\otimes\tau,
\end{equation}
where $\sigma_m(k)$ denotes the restriction of $\sigma_{|_{SU(n)}}(k)$ to the space $\cP^m_{\rm hol}(\bC^n)$ of holomorphic polynomial homogeneous of degree~$m$.

Let $\la_1,\dots,\la_{n-1}$ the simple positive roots in decreasing order. The highest weight $\mu$ is identified by the $(n-1)$-tuple $a=(a_1,\dots,a_{n-1})$ with $a_j=\la_j(\mu)\ge0$. We then write $\mu=\mu_a$.
By Pieri's formula \cite{FH},
$$
\sigma_m\otimes\mu_a=\sum_{b\in B_m}\mu_b,
$$
where $B_m$ is the set of $(n-1)$-tuple $b=(b_1,\dots,b_{n-1})$
with 
$$
b_j=a_j+c_j-c_{j+1}\ , \qquad \sum_1^nc_j=m\ ,\qquad  0\le c_{j+1}\le a_j\text{ for }j=1,\dots,n-1.
$$

Assume that the same $b$ is obtained from two choices $c_1,\dots,c_n$ with $\sum c_j=m$, and $c'_1,\dots,c'_n$  with $\sum c'_j=m'$. This is equivalent to saying that the differences $c_j-c'_j$ do not depend on $j$.

If $\tau$ is singular, at least one of the $a_j$ is 0 and this forces the equalities $c'_j=c_j$ for every $j$,  $m'=m$. Hence $\sigma_{|_{SU(n)}}\otimes \tau$ splits without multiplicities. Moreover, $\tau'$ is also singular and the same conclusion holds for $\sigma'_{|_{SU(n)}}\otimes \tau=(\sigma_{|_{SU(n)}}\otimes \tau')'$.

On the other hand, if $\tau$ is regular, then $a_j\ge1$ for every $j$, and the choices $m=c_j=0$ and $m'=n$, $c'_j=1$ show that $\tau$ is contained both in $\sigma_0\otimes\tau$ and in $\sigma_n\otimes\tau$.
This gives the equivalence between (i) and (ii). 

Assume now that (i) holds. By Proposition \ref{multfreeH_n} and formula \eqref{m-decomposition}, each irreducible component of $\sigma_{|_{SU(n)}}\otimes \tau$ is contained in $\sigma_m\otimes\tau$.
On the other hand, condition (ii) implies that, if $\chi$ is any character of $U(1)$ which extends $\tau$ from the center of $SU(n)$, for every  $F\in L^1_\tau(H_n)$, $\pi_\lambda(F)$ commutes with $\sigma(e^{i\theta}I)\otimes \chi$ for  $\lambda>0$, and with $\sigma'(e^{i\theta}I)\otimes \chi$ for  $\lambda<0$. Since $\sigma(e^{i\theta}I)=e^{im\theta}I$ on $\cP^m_{\rm hol}(\bC^n)$, we have that, denoting by $\tilde\tau$ the representation of $U(n)$ given by
$$
\tilde\tau(e^{i\theta}k)=\chi(e^{i\theta})\tau(k)\ ,\qquad k\in SU(n),
$$
 for every $k\in U(n)$, $\pi_\lambda(F)$ commutes with $\sigma\otimes\tilde\tau$ for $\lambda>0$ and with $\sigma'\otimes\tilde\tau$ for $\lambda<0$. Hence $F\in L^1_{\tilde\tau}(H_n)$ and, in particular,
 $$
 F(e^{i\theta}z,t)=\tilde\tau(e^{i\theta}I)F(z,t)\tilde\tau(e^{-i\theta}I)=F(z,t).
 $$

This argument can be easily reversed to give the opposite implication.

Finally, if $H=\bC^n$ and $K=SU(n)$, the stabilizer of any $z\ne0$ is $SU(n-1)$. 

Assume that the highest weight $\mu$ of $\tau$ is regular, so that $a_j=\la_j(\mu)>0$ for every $j$. Let $\tau^\sharp$ be the extension of $\tau$ to $U_n$ with highest weight 
$$
\mu^\sharp=(\mu_1,\mu_2,\dots,\mu_{n-1},0)=(a_1+\cdots+a_{n-1}\,,\,a_2+\cdots+a_{n-1}\,,\,\dots\,,\,a_{n-1},0)\ .
$$

Another application of Pieri's formula, cf. \cite[p. 80]{FH}, shows that $\tau^\sharp_{|_{U(n-1)}}$ contains both representations with highest weights
$$
(\mu_1,\mu_2,\dots,\mu_{n-1})\ ,\qquad (\mu_1-1,\mu_2-1,\dots,\mu_{n-1}-1)\ ,
$$
which have equivalent restrictions to $SU(n-1)$. By Theorem \ref{eqn R^n gelfand pair}, the triple is not commutative for such $\tau$. 

Conversely, assume that the highest weight of $\tau $ is singular. Given two functions $F_1,F_2\in L^1_\tau(\bC^n)$, let $G_j(z,t)=F_j(z)\eta(t)$, $j=1,2$, with $\eta\in L^1(\bR)$ with integral 1. Then $G_1,G_2\in L^1_\tau(H_n)$, so that $G_1*G_2=G_2*G_1$. But
$$
\int_\bR G_i*G_j(z,t)\,dt=F_i*F_j(z)\ ,
$$
hence $F_1$ and $F_2$ commute.
\end{proof}

Notice that the implication (i)$\Rightarrow$(iii) holds for every $K\subset U(n)$, with $U(n)$ replaced by $K\,U(1)$.

\section{The case where $H=\bR^n$ or $H_n$. Spherical functions}\label{sect:sphericalRnHn}

Let $(\mathbb{R}^n,K,\tau)$ be a commutative triple. For fixed $\xi\in\mathbb{R}^n$, let $V_\tau=\oplus _{j=1}^{l(\xi)}V_{j,\xi}$ be the multiplicity-free decomposition of $V_\tau$ under the action of $K_{\xi}$.
Let $P_{j,\xi}$ denote the projection onto $V_{j,\xi}$ with respect to this decomposition, and $d_j$ the dimension of $V_{j,\xi}$. 
If $F\in L^1_{\tau}(\mathbb{R}^n)$ then $\widehat{F}\in (C_0)_\tau(\mathbb{R}^n)$, and hence $\widehat{F}(\xi)\in\textup{Hom}_{K_\xi}(V_\tau)$. Therefore  $\widehat{F}(\xi)=\oplus_{j=1}^{l(\xi)}\widehat{F}_j(\xi)P_{j,\xi}$ for some scalars $\widehat{F}_j(\xi)$. Since $\widehat{F*G}=\widehat{F}\widehat{G}$, it follows that, for each $j$, $F\rightarrow \widehat{F}_j(\xi)$ defines a multiplicative linear functional of $L^1_\tau(\mathbb{R}^n)$. 
But 
\begin{eqnarray*}
\widehat{F}_j(\xi)&=&\frac{1}{d_j}\textup{Tr}\big[P_{j,\xi}\widehat{F}(\xi)\big]\\&=&\frac{1}{d_j}\textup{Tr}\bigg[P_{j,\xi}\int_{\mathbb{R}^n}e^{i\xi\cdot x}F(x)dx\bigg]\\&=&\frac{1}{d_j}\int_{\mathbb{R}^n}\textup{Tr}\bigg[e^{i\xi\cdot x}P_{j,\xi}F(x)\bigg]dx\\&=&\frac{1}{d_j}\int_{K}\int_{\mathbb{R}^n}\textup{Tr}\bigg[e^{i\xi\cdot (k\cdot x)}P_{j,\xi}F(k\cdot x)\bigg]dxdk.
\end{eqnarray*}
Since $F(k\cdot x)=\tau(k)F(x)\tau(k)^{-1}$ and $\textup{Tr}(AB)=\textup{Tr}(BA)$, we get
\begin{equation}\label{F_j hat}
\begin{aligned}
\widehat{F}_j(\xi)&=\frac{1}{d_j}\int_{K}\int_{\mathbb{R}^n}\textup{Tr}\bigg[e^{i\xi\cdot (k\cdot 
x)}\tau(k^{-1})P_{j,\xi}\tau(k)F(x)\bigg]dxdk\\
&=\frac{1}{d_\tau}\int_{\mathbb{R}^n}\textup{Tr}[\Phi_{\xi,j}(-x)F(x)]dx,
\end{aligned}
\end{equation}
 where
\begin{eqnarray}\label{eqn tau spherical function} 
\Phi_{\xi,j}(x)=\frac{d_\tau}
{d_j}\int_K e^{-i\xi\cdot (k\cdot x)}\tau(k^{-1})P_{j,\xi}\tau(k)dk.
\end{eqnarray}
Since $\Phi_{\xi,j}\in L^{\infty}_\tau(\mathbb{R}^n)$, $\Phi_{\xi,j}$ is a bounded $\tau$-spherical function. We call it a {\it Bessel function of type $\tau$}.

\begin{theorem}\label{all tau spherical functions}\quad
\begin{itemize}
\item[\bf(a)] $\{\Phi_{\xi,j}:\xi\in\mathbb{R}^n,j=1,2,\cdots l(\xi)\}$ is the set of all bounded $\tau$-spherical functions.
\item[\bf(b)] Two spherical functions $\Phi_{\xi_1,j_1},\Phi_{\xi_2,j_2}$ coincide if and only if the following conditions hold:
\begin{itemize}
\item[(i)] $\xi_1$ and $\xi_2$ lie on the same $K$-orbit;
\item[(ii)] if $\xi_2=k\xi_1$, then $V_{j_2,\xi_2}=\tau(k)\big(V_{j_1,\xi_1}\big)$ (this condition is independent of the choice of $k$).
\end{itemize}
\end{itemize}
\end{theorem}

\begin{proof}
Let $\Psi $ be a bounded spherical function. By Theorem \ref{positive on nilpotent}, $\Psi $ is of positive type. Applying Theorem \ref{positive-representation} and Proposition \ref{equivalent condition for positive definite proposition}, there is an irreducible unitary representation $\pi$ of $G=K\ltimes\bR^n$ such that $\tau\subset \pi_{|_K}$ (say, with $V_\tau\subset \cH_\pi$, uniquely determined by the multiplicity-free condition) and 
$$
\big\lan\Psi (x)u,v\big\ran =\big\lan\pi(e,x)u,v\big\ran,
$$
for all $u,v\in V_\tau$.

The irreducible unitary representations of $G$ are those induced from irreducible representations of $G_\xi=K_\xi\times\bR\xi$ for some $\xi\in\bR^n$. By Frobenius reciprocity, for a given $\sigma\in\widehat{K_\xi}$ and $\la\in\bR$, $\tau\subset \textup{Ind}_{G_\xi}^G(\sigma\otimes e^{i\la\cdot})$ if and only if $\sigma\subset \tau_{|_{K_\xi}}$. Hence $\Psi $ determines an element $\la\xi\in \bR^n$, unique up to the action of $K$, and a unique invariant subspace $V_{j,\xi}$ of $V_\tau$.

The rest of the proof follows from the properties of  induced representations.
\end{proof}

\begin{corollary}\label{eigenvalues}
Let $\Phi_{\xi,j}$ be one of the spherical functions in \eqref{eqn tau spherical function} and  $D\in\big(\bD(\bR^n)\otimes \textup{\rm End}(V_\tau)\big)^K$,
$$
D=Q(\de_x)\ ,\qquad Q\in \big(\fS(\bR^n)\otimes \textup{\rm End}(V_\tau)\big)^K\ .
$$

Then the eigenvalue $\lambda_D(\Phi_{\xi,j})$ of $\Phi_{\xi,j}$ under the action of $D$ is the eigenvalue of $Q(-i\xi)$ on $V_{j,\xi}$.
\end{corollary}

\begin{proof}
Since $Q(-i\xi)$ commutes with $\tau_{|_{K_\xi}}$,  it is a scalar multiple of the identity, say $\lambda I$, on $V_{j,\xi}$. Hence $Q(-i\xi)P_{j,\xi}=\la P_{j,\xi}$.
Therefore, 
$$
D\big(e^{-i\xi\cdot x}P_{j,\xi}\big)=Q(-i\xi)e^{-i\xi\cdot x}P_{j,\xi}=\lambda e^{-i\xi\cdot x}P_{j,\xi}.
$$

By the invariance of $D$, the same holds for $e^{-i\xi\cdot (k\cdot x)}\tau(k^{-1})P_{j,\xi}\tau(k)$ and all $k\in K$. Hence, by~\eqref{eqn tau spherical function}, 
$$
D\Phi_{\xi,j}=\lambda\Phi_{\xi,j}. \qquad\qedhere
$$
\end{proof}

\medskip

Consider now a commutative triple $(K\ltimes H_n, K,\tau)$ with $K\subseteq U(n)$ and $\tau\in \widehat K$. Let 
$$
\cP_{\textup{hol}}(\bC^n)\otimes V_\tau=\sum_\alpha W^+_\alpha\ ,\qquad \cP_{\textup{hol}}(\bC^n)\otimes V_\tau=\sum_\beta W^-_\beta
$$
be the decompositions into irreducible $K$-invariant subspaces under $\sigma_{|_K}\otimes \tau$ and $\sigma'_{|_K}\otimes \tau$ respectively.

Since, for $\la\ne0$, $K_{\pi_\lambda}=K$, the general formula \eqref{formula of tau spherical function general case} for spherical functions gives 
\begin{equation}\label{Phi+,Phi-}
\begin{aligned}
\Phi^+_{\lambda,\alpha}(z,t)&=\frac{d_\tau}{d_\alpha}e^{i\lambda t}\textup{Tr}_{\cP_{\textup{hol}}}\big((\pi_\lambda(z,0)\otimes I)P_\alpha\big)\\ 
\Phi^-_{\lambda,\beta}(z,t)&=\frac{d_\tau}{d_\beta}e^{i\lambda t}\textup{Tr}_{\cP_{\textup{hol}}}\big((\pi_\lambda(z,0)\otimes I)P_\beta\big)\ ,
\end{aligned}
\end{equation}
for $\lambda>0$ and $\lambda<0$ respectively.  We call these the {\it Laguerre-type $\tau$-spherical functions}, because the matrix entries of $\pi_\lambda(z,0)$ contain Laguerre functions in $|z|^2$, cf. \cite{Fo}.

 To the Laguerre $\tau$-spherical functions one should add those corresponding to the 1-dimensional representations of $H_n$ (i.e., to $\lambda=0$), which are Bessel functions of type $\tau$.

\begin{theorem}
The bounded spherical functions for $(K\ltimes H_n, K,\tau)$ are the functions $\Phi^+_{\lambda,\alpha}$ and $\Phi^-_{\lambda,\beta}$ in \eqref{Phi+,Phi-} and the  functions
$\Phi^0_{\xi,j}(z,t)=\Phi_{\xi,j}(z)$, where the $\Phi_{\xi,j}$ are the Bessel functions of type $\tau$ \eqref{eqn tau spherical function} for the commutative triple $(K\ltimes\bC^n,K,\tau)$.
\end{theorem}

\section{An example}\label{sect:example}
\vskip.5cm

We give the explicit form of generating differential operators and of the spherical functions in the case

$$
H=\bR^n\ ,\qquad K=SO(n)\ ,\qquad V_\tau=\bC^n,
$$
where  $\tau$ is the defining $n$-dimensional representation of $SO(n)$.

Let $\xi\in\bR^n\setminus\{0\}$. It $n=2$, $K_\xi=\{I\}$. Therefore $V_\tau$ decomposes with multiplicity under the action of $K_\xi$. Hence, if $n=2$, $(\bR^n,K,\tau)$ is not a commutative triple. 

On the other hand, for $n\geq 3$ and $\xi=\rho e_1\ne0$, $K_\xi\cong SO(n-1)$ and $V_\tau$ decomposes under the action of $K_\xi$ as
\begin{equation}\label{V-splitting} 
V_\tau=\begin{cases}\bC e_1\oplus\bC(e_2+ie_3)\oplus\bC(e_2-ie_3)=V_{e_1,1}\oplus V_{e_1,2}\oplus V_{e_1,3}&\textup{ if }n=3,\\
\bC e_1\oplus e_1^\perp=V_{e_1,1}\oplus V_{e_1,2}&\textup{ if }n>3,
\end{cases}
\end{equation} 
in both cases without multiplicities. 
Hence for any $n\geq 3$, $(\bR^n,K,\tau)$ is a commutative triple. 

We describe a system $\cD$ of generators of $\big(\mathbb{D}(\bR^n)\otimes \textup{End}(V_\tau)\big)^K$. To this purpose, we investigate the structure of the algebra of $K$-invariant elements in  $\fS(\bR^n)\otimes \textup{End}(V_\tau)$.

We denote by $\sigma$ the representation of $K$ on the space $\cP(\bR^n,M_{n,n})\cong\fS(\bR^n)\otimes \textup{End}(V_\tau)$ of matrix-valued polynomials, given by 
$$
\big(\sigma(k)P\big)(x)=kP(k\inv x)k^{-1}.
$$
For a $\sigma$-invariant space $W\subset \cP(\bR^n, M_{n,n})$, let $W^\tau$ denote the space of all $\sigma(k)$-fixed elements in $W$. Note that a polynomial $P:\bR^n\rightarrow M_{n,n}$ is $\sigma$-invariant iff  
$$
P(k\cdot x)=kP(x)k^{-1}\qquad\forall\,k\in K.
$$
Let $\cI=\bC\big[|x|^2\big]$ be the space of all $K$-invariant scalar-valued polynomials on $\bR^n$, $\cH(\bR^n)$  the space of all harmonic polynomials, $\cH_m(\bR^n)$ be the space of all homogeneous harmonic polynomials of degree $m$. Then
\begin{equation}\label{a.1}
\begin{aligned}
\big(\cP(\bR^n,M_{n,n})\big)^\tau 
&=\cI\,\big(\cH(\bR^n,M_{n,n})\big)^\tau
\\ &= \cI\,\bigoplus_{m=0}^\infty\big(\cH_m(\bR^n,M_{n,n})\big)^\tau
\end{aligned}
\end{equation} 

Let $\delta_m$ and $\pi\sim\tau\otimes\tau'$ denote the natural representations of $K$ on $\cH_m(\bR^n)$ and $M_{n,n}$ respectively:
$$
\delta_m(k)h(x)=h(k\inv x)\ ,\qquad \pi(k)A=kAk\inv\ ,
$$
so that the restriction of $\sigma$ to $\cH_m(\bR^n,M_{n,n})$ is equivalent to $\delta_m\otimes\pi$.

Under $\pi$, $M_{n,n}$ decomposes into irreducibles as:
$$
M_{n,n}=\bC\,I\oplus M_{n,n}^{\textup{skew}}\oplus 
M_{n,n}^{\textup{sym},0},
$$
where $M_{n,n}^{\textup{skew}}$ is the space of  $n\times n$ skew-symmetric matrices, and $M_{n,n}^{\textup{sym},0}$ is the space of  $n\times n$ symmetric matrices with zero trace. Then, for every $m$,
\begin{equation}\label{a.2}
\big (\cH_m(\bR^n,M_{n,n})\big)^\tau =\big(\cH_m(\bR^n,\bC\,I)\big)^\tau\bigoplus \big(\cH_m(\bR^n M_{n,n}^{\textup{skew}})\big)^\tau\bigoplus \big(\cH_m(\bR^n M_{n,n}^{\textup{sym},0}) \big)^\tau.  
\end{equation}

Let us introduce the following elements of $\big(\cP(\bR^n)\otimes\textup{End}(V)\big)^\tau$.
$$
I_{3,1}(x)=
\begin{pmatrix}
0 & -x_3 & x_2\\
x_3 & 0 & -x_1\\
-x_2 & x_1 & 0
\end{pmatrix}
\in \big(\cH_1(\bR^3, M_{3,3}^{\textup{skew}})\big)^\tau,
$$
$$
I_{n,2}(x)=
x\trans x-\frac{1}{n}|x|^2I
\in \big(\cH_2(R^n), M_{n,n}^{\textup{sym},0}\big)^\tau)\qquad n\geq 3.
$$

\begin{theorem}\label{Js}
Let $P\in \big(\cP(\bR^n,M_{n,n})\big)^\tau$. Then
$$
P(x)=\begin{cases}p_0\big(|x|^2\big)+p_1\big(|x|^2\big)J_{3,1}(x)+p_2\big(|x|^2\big)J_{3,2}(x)&\text{ if }n=3\\
p_0\big(|x|^2\big)+p_2\big(|x|^2\big)J_{n,2}(x)&\text{ if }n>3
\end{cases}
$$
for some $p_0,p_1,p_2\in\cI$.
\end{theorem}

\begin{proof}
Clearly,
$$
\big(\cH(\bR^n,\bC\,I)\big)^\tau=\big(\cH(\bR^n)\cap\cI)\otimes I=\bC\, I.
$$

Let $U_1=M_{n,n}^{\textup{skew}}$, $U_2=M_{n,n}^{\textup{sym},0}$. By Schur's lemma it follows that $\big(H_m(\bR^n)\otimes U_j\big)^\tau$ contains a non-zero element iff $\delta_m\sim\pi^\prime_{|_{U_j}}\sim\pi_{|_{U_j}}$; and in that case it is unique up to a scalar multiple.

A simple check of dimensions shows that this equivalence can only hold for $j=m=2$ if $n\ge3$, and for $j=m=1$ if $n=3$.
\end{proof}

Theorem \ref{Js} implies that $\big(\cP_K(\bR^n)\otimes\textup{End}(V)\big)^\tau$ is generated  as an $\cI$-module by $I,I_{3,1}, I_{3,2}$ if $n=3$, and by $I, I_{n,2}$ if $n>3$. 

\begin{corollary}
$\{I,I_{3,1}\}$
(resp. $\{I, I_{n,2}\}$) is a set of independent generators for the algebra $\big(\cP(\bR^n)\otimes \textup{End}(V)\big)^\tau$ when $n=3$ (resp. $n>3$).

For $F:\bR^n\longrightarrow \bC^n$, let 
$$
\textbf{D}F=\begin{cases}I_{3,1}(\de_x)F=\textup{\rm curl} F& \text{ if }n=3\\
(I_{n,2}(\de_x)+\frac{1}{n}\Delta I)F=\textup{\rm grad}\,\textup{\rm div}F&\text{ if }n>3.
\end{cases}
$$ 

Then $\cD=\{\Delta I,\textbf{D}\}$ generates the algebra $\big(\mathbb{D}(\bR^n)\otimes \textup{End}(V_\tau)\big)^K$.
\end{corollary}

\begin{proof}
We just need to mention that, for $n=3$,  $I_{3,2}=|x|^2I+I_{3,1}^2$ (equivalently, $\textup{\rm grad}\,\textup{\rm div}F= \Delta F-\textup{\rm curl}\,\textup{\rm curl}F$). Moreover,
$I,I_{3,1}$ are algebraically independent when $n=3$ (resp. $I,I_{n,2}$ when $n>3$).
\end{proof}
\medskip

According to Theorem \ref{all tau spherical functions}, the bounded spherical function are given by formula \eqref{eqn tau spherical function}. 

For $\xi=0$, we have a unique spherical function $\Phi_{0,1}(x)=I$. We can then take $\xi=s e_1$ with $s>0$. By \eqref {eqn tau type function}, it suffices to compute $\Phi_{s e_1,j}=\Phi_{s,j}$ for $x=re_1$ with $r>0$, as we already know that $\Phi_{s,j}(0)=I$.

Factoring the integral in \eqref{eqn tau spherical function} modulo $K_\xi=\Big\{\begin{pmatrix} 1&0\\0&h\end{pmatrix}:h\in SO(n-1)\Big\}$, we reduce ourselves to an integral over the sphere $S^{n-1}$,
$$
\Phi_{s ,j}(re_1)=\frac n{d_j}\int_{S^{n-1}}e^{is r y\cdot e_1}P_{j,y}\,d\sigma_{n-1}(y)\ ,
$$
where $\sigma_{n-1}$ is the normalized surface measure on $S^{n-1}$, $y=k\inv e_1$ and 
$P_{j,y}=k\inv P_{j,e_1}k$.
\medskip

\subsubsection{The case $n>3$}\quad
\smallskip

Referring to the decomposition \eqref{V-splitting}, we have
$$
P_{1,y}=k\inv e_1\trans e_1 k=y\trans y\ ,\qquad P_{2,y}=I_n-P_{1,y}=I_n-y\trans y.
$$

Hence, for each $s>0$ we have two spherical functions,
$$
\begin{aligned}
\Phi_{s,1}(re_1)&=n\int_{S^{n-1}}e^{is r y\cdot e_1}y\trans y\,d\sigma(y)\\
 \Phi_{s,2}(re_1)&=\frac n{n-1}\int_{S^{n-1}}e^{is r y\cdot e_1}(I_n-y\trans y)\,d\sigma(y)=\frac n{n-1}\Big(\int_{S^{n-1}}e^{is r y\cdot e_1}\,d\sigma(y)\Big)I_n-\frac1{n-1}\Phi_{s,1}(re_1)\ .
 \end{aligned}
$$

Setting $y=(\cos t, (\sin t) z)$, with $z\in S^{n-2}$,
$$
\Phi_{s,1}(re_1)=nc_n\inv\int_0^\pi\int_{S^{n-2}} e^{is r \cos t}\begin{pmatrix}\cos^2t&(\cos t\sin t) \trans z\\ (\cos t\sin t )z&(\sin^2 t)z\trans z\end{pmatrix}\,d\sigma_{n-2}(z)\,\sin^{n-2}t\,dt,
$$
with
$$
c_n =\int_0^\pi\sin^{n-2}t\,dt=  \pi^{\frac12}\frac{\Gamma(\frac {n-1}2)}{\Gamma(\frac n2)}.
$$

Integration in $z$ annihilates the off-diagonal terms $z_j$ and $z_jz_k$ with $j\ne k$, while $\int_{S^{n-2}}z_j^2 d\sigma_{n-2}(z)=\frac1{n-1}$ for symmetry reasons. Hence
$$
\begin{aligned}
\Phi_{s,1}(re_1)&=  nc_n\inv\int_0^\pi e^{is r \cos t}\begin{pmatrix}\cos^2t&0\\ 0&\frac{\sin^2 t}{n-1}I_{n-1}\end{pmatrix}\,\sin^{n-2}t\,dt\\
&= nc_n\inv\int_0^\pi e^{is r \cos t}\sin^{n-2}t\,dt\begin{pmatrix}1&0\\ 0&0\end{pmatrix}\\
&\qquad + nc_n\inv\int_0^\pi e^{is r \cos t}\sin^nt\,dt\begin{pmatrix}-1&0\\ 0&\frac1{n-1}I_{n-1}\end{pmatrix}.
\end{aligned}
$$

From the identity
$$
\int_0^\pi e^{iu \cos t}\sin^kt\,dt=2^{\frac k2}\pi^\half \Gamma\Big(\frac {k+1}2\Big)\frac{J_{\frac k2}(u)}{u^{\frac k2}}
$$
we obtain that
$$
\Phi_{s,1}(re_1)=n2^{\frac n2-1}\Gamma\Big(\frac n2\Big)\bigg(\frac{J_{\frac n2-1}(s r)}{(s r)^{\frac n2-1}}\begin{pmatrix}1&0\\ 0&0\end{pmatrix}+\frac{J_{\frac n2}(s r)}{(s r)^{\frac n2}}\begin{pmatrix}-(n-1)&0\\ 0&I_{n-1}\end{pmatrix}\bigg).
$$

For general $x=rke_1$, we  have
$$
\begin{aligned}
k\begin{pmatrix}1&0\\ 0&0\end{pmatrix}k\inv&=(ke_1)\trans (ke_1)=|x|^{-2}x\trans x,\\
k\begin{pmatrix}0&0\\ 0&I_{n-1}\end{pmatrix}k\inv&=I_n-|x|^{-2}x\trans x,
\end{aligned}
$$
so that
$$
\Phi_{s,1}(x)=n2^{\frac n2-1}\Gamma\Big(\frac n2\Big)\bigg(\frac{J_{\frac n2-1}(s |x|)}{s^{\frac n2-1}|x|^{\frac n2+1}}x\trans x+\frac{J_{\frac n2}(s |x|)}{s^{\frac n2}|x|^{\frac n2+2}}\big(|x|^2I_n-nx\trans x\big)  \bigg)
$$

Similarly,
$$
\Phi_{s,2}(re_1)=\frac{n2^{\frac n2-1}}{n-1}\Gamma\Big(\frac n2\Big)\bigg(\frac{J_{\frac n2-1}(s |x|)}{s^{\frac n2-1}|x|^{\frac n2+1}}(n|x|^2I_n-x\trans x)-\frac{J_{\frac n2}(s |x|)}{s^{\frac n2}|x|^{\frac n2+2}}\big(|x|^2I_n-nx\trans x\big)  \bigg).
$$

\smallskip
By Corollary \ref{eigenvalues}, the map $\rho_\cD$ of Section \ref{sect:embeddings} relative to the system $\cD=(\Delta I,\textup{\rm grad}\,\textup{\rm div})$
is given by
$$
\rho_\cD(\Phi_{s,1})=(-s^2,-s^2)\ ,\qquad \rho_\cD(\Phi_{s,2})=\big(-s^2,0\big).
$$

\medskip

\subsubsection{The case $n=3$}\quad
\smallskip

For $y\in S^2$, we have
$$
\begin{aligned}
P_{1,y}&=y\trans y\\
 P_{2,y}&=\frac{1}{2}k\inv(e_2+ie_3)\trans(e_2-ie_3)k=\frac{1}{2}\big(I_3-y\trans y+iI_{3,1}(y)\big)\\
  P_{3,y}&=\frac{1}{2}k\inv(e_2-ie_3)\trans(e_2+ie_3)k=\frac{1}{2}\big(I_3-y\trans y-iI_{3,1}(y)\big).
  \end{aligned}
$$

For each $s>0$ we now have three spherical functions,
$$
\begin{aligned}
\Phi_{s,1}(re_1)&=3\int_{S^2}e^{is r y\cdot e_1}y\trans y\,d\sigma(y)\\
 \Phi_{s,2}(re_1)&=\frac{3}{2}\int_{S^2}e^{is r y\cdot e_1}\big(I_n-y\trans y+iI_{3,1}(y)\big)\,d\sigma(y)\\
  \Phi_{s,3}(re_1)&= \frac{3}{2}\int_{S^2}e^{is r y\cdot e_1}\big(I_n-y\trans y-iI_{3,1}(y)\big)\,d\sigma(y).
 \end{aligned}
$$

The only new computation that is needed concerns the integral
$$
\cI(u)=\int_{S^2}e^{iu y\cdot e_1}I_{3,1}(y)\,d\sigma(y).
$$

In the polar coordinates
$y=(\cos t,\sin t\cos\ph,\sin t\sin\ph)$, this becomes
$$
\begin{aligned}
\cI(u)&=\frac1{4\pi}\int_0^\pi\int_0^{2\pi}e^{iu\cos t}\begin{pmatrix} 0&-\sin t\sin\ph&\sin t\cos\ph \\ \sin t\sin\ph &0&-\cos t \\-\sin t\cos\ph &\cos t&0\end{pmatrix}\sin t\,d\varphi\,dt\\
&=\half \Big(\int_0^\pi e^{iu\cos t}\cos t\sin t\,dt\Big)I_{3,1}(e_1)\\
&=\sqrt\frac\pi2\frac{J_\frac32(u)}{u^\half}\,iI_{3,1}(e_1).
\end{aligned}
$$

Proceeding as for the case $n>3$, we obtain
$$
\begin{aligned}
\Phi_{s,1}(x)&=3\sqrt\frac\pi2\bigg(\frac{J_\half(s |x|)}{s^\half |x|^{\frac 52}}x\trans x+\frac{J_{\frac 32}(s |x|)}{s^{\frac 32}|x|^{\frac 72}}\big(|x|^2I_3-3x\trans x\big)  \bigg)\\
\Phi_{s,2}(x)&=\frac{3}{2}\sqrt\frac\pi2\bigg(\frac{J_\half(s |x|)}{s^{\half}|x|^{\frac 52}}(3|x|^2I_3-x\trans x)-\frac{J_{\frac 32}(s |x|)}{s^{\frac 32}|x|^{\frac 72}}\big(|x|^2I_3-3x\trans x+s|x|^2 I_{3,1}(x)\big)  \bigg)\\
\Phi_{s,2}(x)&=\frac{3}{2}\sqrt\frac\pi2\bigg(\frac{J_\half(s |x|)}{s^{\half}|x|^{\frac 52}}(3|x|^2I_3-x\trans x)-\frac{J_{\frac 32}(s |x|)}{s^{\frac 32}|x|^{\frac 72}}\big(|x|^2I_3-3x\trans x-s|x|^2 I_{3,1}(x)\big)  \bigg).
\end{aligned}
$$

Taking $\cD=(\Delta I, \textup{\rm curl} )$, we have
$$
\rho_\cD(\Phi_{s,1})=(-s^2,0)\ ,\qquad \rho_\cD(\Phi_{s,2})=\big(-s^2,-s\big)\ ,\qquad \rho_\cD(\Phi_{s,3})=\big(-s^2,s\big).
$$


\begin{thebibliography}{99}
\bibitem{ADR}
	Astengo, F., Di Blasio, B., Ricci, F., \textit{Gelfand transforms of polyradial Schwartz functions on the Heisenberg group},
	J. Funct. Anal. 251 (2007), 772--791.


\bibitem{ADR1}
	Astengo, F., Di Blasio, B., Ricci, F., \textit{Gelfand pairs  on the Heisenberg group of and Schwartz functions},
	J. Funct. Anal. 256  (2009), 1565--1587.

\bibitem{B} Benson, C., Jenkins, J., Ratcliff, G., \textit{On Gelfand pairs associated with solvable Lie groups,} Trans. Amer. Math. Soc. 321 (1990), 187--214.

\bibitem{bre} Bredon, G.E., \textit{Introduction to compact transformation groups,} Acad. Press (1972).

\bibitem{C} Camporesi, R., \textit{The spherical transform for homogeneous vector bundles over Riemannian symmetric spaces,} Journal of Lie Theory 7 (1997), 29--60.

\bibitem{Car} Carcano, G., \textit{A commutativity condition for the algebra of invariant functions,} Boll. Un. Mat. Ital. 7 (1987), 1091--1105.


\bibitem{D} Deitmar, A., \textit{Invariant operators on higher K -types,} J. Reine Angew. Math. 412 (1990), 97--107.

\bibitem{F} Ferrari Ruffino, F., \textit{The topology of the spectrum for Gelfand pairs on Lie groups,} Boll. Unione Mat. Ital. Sez. B Artic. Ric. Mat. (8) 10 (2007) 569--579.
		 	
	\bibitem{FRY1}
		Fischer, V., Ricci, F., Yakimova, O.,
		\textit{Nilpotent Gelfand pairs and spherical transforms of Schwartz functions I. Rank-one actions on the centre}, Math. Zeitschr. 271\,(2012) 221--255. 

	\bibitem{FRY2}
		Fischer, V., Ricci, F., Yakimova, O.,
		\textit{Nilpotent Gelfand pairs and spherical transforms of Schwartz functions II. Taylor expansions on singular sets}, in  \textit{Lie Groups: Structure, Actions and Representations}, BirkhŠuser (2013) 81-112.
		
		\bibitem{FRY}
Fischer, V., Ricci, F., Yakimova, O., \textit{Nilpotent Gelfand pairs and spherical transforms of Schwartz functions III. Isomorphisms between Schwartz spaces under Vinberg's condition}, arXiv:1210.7962.


\bibitem{Fo} Folland, G., \textit{Harmonic Analysis in Phase Space,} Princeton Univ. Press, Princeton, 1989.

\bibitem{FH} Fulton, W., Harris, J., \textit{Representation Theory, A First Course,} Springer, 1991.

\bibitem{GV} Gangolli, R., Varadarajan, V.S., \textit{Harmonic Analysis of Spherical Functions on Real Reductive Groups,} Springer, 1988.

\bibitem{G} Godement, R., \textit{A theory of spherical functions. I,} Trans. Amer. Math. Soc. 73 (1952),  496--556.


\bibitem{H} Helgason, S., \textit{Groups and Geometric Analysis,} Academic Press, New York, 1984.

\bibitem{M} Minemura, K., \textit{Invariant differential operators and spherical sections on a homogeneous vector bundle,} Tokyo J. Math 15 (1992), 231--245.

\bibitem{N} Naimark, M. A., \textit{Normed rings,} Wolters-Noordhoff, 1959.

\bibitem{P} Poguntke, D., \textit{Rigidly symmetric $L^1$-group algebras,} Sem. Sophus Lie 2 (1992), no. 2, 189--197. 

\bibitem{S} Sakai, S., \textit{On the representations of semi-simple Lie groups,} Proc. Japan Acad. 130 (1954), 14--18.

\bibitem{T} Thomas, E. G. F.,  \textit{An infinitesimal characterization of Gelfand pairs,} Proc. Conference on Harmonic Analysis and Probability in honour of Shizuo Kakutani, Yale University, 1982.


\bibitem{Ti} Tirao, J. A., \textit{Spherical functions,} Rev. Un. Mat. Arg. (1977), 75--92. 


\bibitem{V} Vinberg, E. B., \textit{Commutative homogeneous 
	spaces and co-isotropic symplectic actions},
	Russian Math. Surveys 56 (2001), 1--60.


\bibitem{W} Wallach, N. R., \textit{Harmonic analysis on homogeneous spaces,} Marcel Dikker New York, 1973.

\bibitem{War} Warner, G., \textit{Harmonic analysis on semi-simple Lie groups} I, II, Springer Verlag, 1972.

\bibitem{Wo} Wolf, J., \textit{Harmonic analysis on commutative spaces,} AMS Math. Surveys and Monographs, 2007.

\bibitem{Y} Yakimova, O., \textit{Principal Gelfand pairs},
	Transf. Groups 11 (2006), 305--335.


\end{thebibliography}
\end{document}